\documentclass[12pt]{article} 
\usepackage{amsmath}
\usepackage{amsfonts}
\usepackage{color}
 \usepackage[latin1]{inputenc}

\usepackage[pdftex]{hyperref}
\usepackage[pdftex]{graphicx}

\usepackage{amssymb}   
\usepackage{amsthm} 
\usepackage{verbatim} 
\numberwithin{equation}{section}

\newtheorem{lemma}[equation]{Lemma}
\newtheorem{proposition}[equation]{Proposition}
\newtheorem{corollary}[equation]{Corollary}
\newtheorem{remark}[equation]{Remark}
\newtheorem{theorem}[equation]{Theorem}
\newtheorem{conjecture}[equation]{Conjecture}

\def\L{\mathcal{L}}
\def\O{\mathcal{O}}
\def\X{\mathcal{X}}
\def\Y{\mathcal{Y}}
\def\C{\mathbb{C}}
\def\P{\mathbb{P}}

\def\N{\mathcal N}
\def\F{\mathbb{F}}

\begin{document}
\title  {Homogeneous interpolation on ten points}
\author {Ciro Ciliberto and Rick Miranda}

\date{}
\maketitle

\begin{abstract}
In this paper we prove that for all pairs $(d,m)$ with $d/m\geq 174/55$,
the linear system of plane curves of degree $d$
with ten general base points of multiplicity $m$
has the expected dimension.
\end{abstract}

\tableofcontents

\section*{Introduction}

Let us denote by $\L(d;m_1^{s_1},\ldots,m_h^{s_h})$
the linear system of plane curves of degree $d$
with $s_i$ general points of multiplicity at least $m_i$, $i=1,...,h$.
This linear system has been under study for more than a century,
and in general the computation of its dimension, which is the
basic problem in multivariate Hermite interpolation,
is still an open problem.

For nine or fewer points the dimension is known,
which is a classical result going back to Castelnuovo \cite {castelnuovo}
(see also \cite{nagata2} and \cite{Har2}).
There is a set of related conjectures about this dimension due to
Segre, Harbourne, Gimigliano and Hirschowitz
(see \S \ref{sec:GHH}) and Nagata (see \cite {nagata}),
which we will review in the next section.
However, even in the \emph{homogeneous} case with ten points,
i.e. for the linear system  $\L(d;m^{10})$,
the dimension is not known; it is expected to be
$$
e(\L(d;m^ {10}))=\max\{-1, \frac{d(d+3)}{2} -5m(m+1)\}
$$
and the conjectures, in this case, assert that the dimension is as expected.
In this paper we prove the following result which supports the above conjectures. 

\begin{theorem}\label{thm:main}
For every pair of integers $(d,m)$ with $d/m\geq {174}/{55}$,
the linear system $\L(d;m^{10})$ has the {expected dimension}.
\end{theorem}

Our proof relies on a fairly complete understanding of linear systems
on surfaces supporting an anticanonical curve;
in \S\ref{sec:GHH} we review the results we need,
whereas we devote \S \ref {sec:nagata}
to the connection between the conjecture of
Segre, Harbourne, Gimigliano and Hirschowitz
and the one of Nagata.

We note that the problem of computing the dimension of these linear systems
may be formulated on the blow--up of the plane at the multiple points;
if $E_i$ is the blow--up of the point $p_i$,
then the system $\L(d;m_1^{s_1},\ldots,m_r^{s_r})$ of plane curves
corresponds to $H^0(B,\O_B(dH-\sum_i s_iE_i))$
where $B$ is the blow--up of the plane,
and $H$ is the pullback of the line class.
In terms of the cohomology of this sheaf,
the system is non--empty if $H^0 \neq 0$,
and has the \emph{expected dimension} if $H^1 =0$: we say in this case
that the system is \emph{non--special}. 

In \S\ref{sec:firstdeg} we present a degeneration of the blown--up plane
which we will use to begin the analysis of the $H^0$. 
This degeneration has been introduced in \cite{CM98} and \cite{CM00}
for similar purposes,
and we will review it briefly.

Sections \ref {sec:throw}, \ref {sec:computation}  and \ref {sect:gentrasv}
introduce technical tools that are used in the remaining
sections, in which we discuss refinements of the degeneration,
necessary to prove the theorem under various
hypotheses for the ratio $d/m$. The proof eventually consists in 
showing non--speciality in certain limit situations.
The degenerations that we introduce can be considered to be inspired
by the minimal model program in birational geometry;
we perform explicit modifications (to the central fiber of our degeneration)
in order to make the limit bundle more nef.

The basic technique
can in principle be applied even beyond the bound $174/55$. 
This would require however the understanding of more and
more complicated degenerations, which, at the moment,
seem difficult to handle.

A corollary of our theorem is that
the linear system $\L(174;55^{10})$ is empty, 
which itself has been an open problem for about twenty years
(``le cas inviol\'e" according to A. Hirschowitz \cite {Hirsh}).
In conclusion, we would like to stress that,
more than this specific result,
it is the success of the technique we introduce here 
which provides the strongest evidence to date for the truth of the conjecture,
since, as we said, there seems to be no theoretical obstruction to carrying it further,
but only computational
complications.

\section{The Segre--Harbourne--Gimigliano--Hirschowitz conjecture}
\label {sec:GHH}

Let $p_1,...,p_k$ be points of the complex projective plane $\P^2$.
We will suppose that they are distinct or {\it infinitely near},
i.e., we have a sequence of morphisms
$f_i: S_i\to S_{i-1}$, $i=1,...,k$,
with $S_0=\P^ 2$ and $f_i: S_i\to S_{i-1}$
the blow--up of $S_{i-1}$ at a point $p_i$.
We will set $\bar\P=S_k$.
On $\bar \P$ we have \emph{$(-1)$--cycles} $E_1,...,E_k$
which are the pull--backs to $\bar \P$ of the points $p_1,...,p_k$,
and the class $H$ which is the pull-back of a line in $\P^2$.
Note that $H, E_1,...,E_k$
freely generate ${\rm Pic}(\bar \P)$. 

Let $d, m_1,...,m_k$ be non--negative integers.
We will denote by $\L(d;m_1p_1,...,m_kp_k)$
the complete linear system $\vert dH-\sum_{i=1}^k m_iE_i\vert$ on $\bar \P$.
We will use the same notation to denote the corresponding line bundle on $\bar \P$,
as well as the push--forward of $\vert dH-\sum_{i=1}^ km_iE_i\vert$ to $\P^ 2$,
i.e. the  linear system of plane curves of degree $d$
with multiplicity at least $m_i$ at $p_i$, $i=1,...,k$.  
Notice that $\vert-K_{\bar \P}\vert =\L(3;1^ k)$. 

The \emph{virtual dimension} of $\L(d;m_1p_1,...,m_kp_k)$ is
\[
v=v(\L(d;m_1p_1,...,m_kp_k))=\frac{d(d+3)}{2}-\sum_{i=1}^k \frac {m_i(m_i+1)}{2}
\]
and the \emph{expected dimension} is
\[
e(\L(d;m_1p_1,...,m_kp_k))
=\min \{ -1, v\}.
\]
One has 
\[
\dim(\L(d;m_1p_1,...,m_kp_k) \geq e(\L(d;m_1p_1,...,m_kp_k))
\]
and
$\L(d;m_1p_1,...,m_kp_k)$ is said to be \emph{special}
if strict inequality holds.
A special system $\L(d;m_1p_1,...,m_kp_k)$ is not empty,
and $h^1(\bar\P, \L(d;m_1p_1,...,m_kp_k))>0$.
We note that $h^2(\bar \P, \L(d;m_1p_1,...,m_kp_k))=0$ if, as we will assume, $d\geq 0$.

The dimension of $\L(d;m_1p_1,...,m_kp_k)$
is upper semicontinuous in the position of the points $p_1,...,p_k$.
Therefore one may expect to have special systems
for special positions of $p_1,...,p_k$, which is the case. 

By contrast, if the points $p_1,...,p_k$
are distinct and sufficiently general in $(\P^2)^k$
then the dimension of $\L(d;m_1p_1,...,m_kp_k)$ reaches a minimum.
When $p_1,...,p_k$ are general we will denote $\L(d;m_1p_1,...,m_kp_k)$
simply by $\L(d;m_1,...,m_k)$ and we will use the notation 
$\L(d;m_1^ {s_1},...,m_h^ {s_h})$ for repeated multiplicities.
A naive conjecture would be that this minimum dimension
coincides with the expected dimension.
This is well known to be false; 
a source of counterexamples is the following.
Suppose $\L=\L(d;m_1,...,m_k)$
is not empty.
Let $E$ be a smooth rational curve on $\bar \P$ with $E^ 2=-h$,
i.e. a \emph{$(-h)$--curve}. 
Assume that $E$ is a $(-1)$--curve and that $E\cdot \L = -n$ with $n \geq 2$.
Then $nE$ sits in the base locus of $\L$
and $h^1(\bar\P, \L) \geq \binom{n}{2}$.
A system $\L$ like this is said to be \emph{$(-1)$--special}.
Examples of this sort are $\L(2;2^2)$, $\L(4;2^5)$ etc. 

We can state now the following conjecture due to Harbourne, Gimigliano and Hirschowitz
(see  \cite{Har}, \cite {Gi}, \cite{Hirsh}).

\begin{conjecture}[HGH] \label{conj:GHH}
A system $\L(d;m_1,...,m_k)$ is special if and only if it is $(-1)$--special.
\end{conjecture} 

There is an earlier conjecture due to B. Segre (see \cite {Seg}),
which, as shown in \cite{cmsegre}, is equivalent to the previous one. 

\begin{conjecture} [Segre] \label{conj:Segre}
If a system $\L(d;m_1,...,m_k)$ is special then its general member is non--reduced.
\end{conjecture} 

We will refer to either one of the two above conjectures as to the 
Segre--Harbourne--Gimigliano--Hirschowitz (SHGH) conjecture.

In the homogeneous case $m_1=\dots =m_k=m$ with $k\geq 10$, this conjecture
implies (see \cite {CM00}) the following.

\begin{conjecture}  \label{conj:Segrehom} If $k\geq 10$ a system $\L(d;m^ k)$ 
is never special.\end{conjecture} 

It is useful to give a different equivalent formulation
of Conjecture \ref{conj:GHH},
involving Cremona transformations.
Consider a non--empty system $\L(d;m_1p_1,...,m_kp_k)$
and assume that $d\geq m_1\geq m_2\geq ...\geq m_k$.
The system will be said to be \emph{Cremona reducible}
if $m_1+m_2+m_3>d$
and there is an irreducible conic passing through $p_1,p_2,p_3$
(which is certainly the case if $p_1,p_2,p_3$ are distinct and not collinear).
The reason for the name is that
the quadratic transformation based at $p_1,p_2,p_3$ sends
$\L(d;m_1p_1,...,m_kp_k)$ to a new system $\L(d';m'_1p'_1,...,m'_kp'_k)$
with $d'=2d-(m_1+m_2+m_3)<d$,
and the dimension of the two systems are the same.
If $m_1+m_2+m_3\leq d$ the system is said to be {\it standard}.

Note the following result (see \cite{Hirsh}, \cite{Har2}):

\begin{proposition}\label{prop:standard}
A standard system is not $(-1)$--special.
\end{proposition}

Hence an equivalent formulation of Conjecture \ref{conj:GHH}, thus of the SHGH conjecture, is:

\begin{conjecture}\label{conj:GHH2}
A standard system with general base points is not special.
\end{conjecture} 

It  goes back to Castelnuovo
(see \cite{castelnuovo}),
that Conjecture \ref{conj:GHH} holds if $k \leq 9$.
More recent treatments  can be found in
\cite{nagata2}, \cite{Gi}, \cite{Har}, \cite{Har2}.
To be specific, one has the following more general results. 

A standard system $\L=\L(d;m_1p_1,...,m_kp_k)$ is said to be \emph{excellent}
[resp. \emph{almost excellent}]
if $\L \cdot K_{\bar\P} < 0$
[resp.  if $\L \cdot K_{\bar\P}\leq 0$].
In other words, $\L$ is excellent [resp. almost excellent] if 
\[
3d-\sum_{i=1}^k m_i > 0 \quad [{\rm resp.}  \quad 3d-\sum_{i=1}^k m_i \geq 0].
\]
Moreover, one says that we are in the \emph{anticanonical case} 
if there is a curve $D$ in the linear system $\vert -K_{\bar\P}\vert=\L(3;p_1,\ldots,p_k)$.
If $k \leq 9$ one is in the anticanonical case.
If there is a  reduced and irreducible curve $D\in  \vert -K_{\bar\P}\vert$,
one says that we are in the {\it strong anticanonical case}.
Then $p_1,...,p_k$ are smooth points of a reduced, irreducible cubic curve in $\P^ 2$.
If $p_1,...,p_k$ are general points and $k\leq 9$,
then we are in the strong anticanonical case.

Recall that a line bundle $\L$ is \emph{nef} if for every irreducible curve $C$ on $\bar\P$, one has $\L\cdot C\geq 0$. 

For the following result, see \cite{Har2}.

\begin{proposition}\label{prop:strongantiK}
Suppose we are in the strong anticanonical case
and consider a system $\L=\L(d;m_1p_1,...,m_kp_k)$. Then:
\begin{itemize}
\item [(i)] if $\L$ is standard, then $\L$ is \emph{effective}, i.e. $\dim (\L)\geq 0$;
\item [(ii)] $\L$ is almost excellent if and only if it is nef;
\item [(iii)] if $\L$ is excellent, then it is non--special.
\end{itemize}
\end{proposition}

\begin{remark}\label{rem:leq9}
{\rm If $k\leq 9$ and $p_1,...,p_k$ are general points,
then a standard system $\L$ is almost excellent,
and actually excellent, unless $k=9$ and $\L$ is a multiple of the anticanonical system.
Thus the SHGH conjecture
follows by Proposition \ref {prop:strongantiK} in this case.}
\end{remark}

\begin{remark}\label{rem:geq9}
{\rm Assume that $k\geq 9$, $p_1,...,p_k$ are general points,
$m_i \geq m_{i+1}$ for each $i$, and $d\leq m_{k-2}+m_{k-1}+m_k$.
Consider $\L=\L(d;m_1,...,m_k)$ and notice that $\L\cdot K_{\bar \P}\geq 0$.
If $k = 9$ and $\L$ is effective, then one concludes that $\L\cdot K_{\bar \P}=0$,
hence $\L$ is a multiple of the anticanonical system,
and has dimension zero.
If $k\geq 10$ then, by considering only the first $9$ points,
the $k=9$ analysis shows that $\L$ is empty.}
\end{remark}

This last remark gives a first result:

\begin{corollary}\label{cor:ratio<=3}
If $d/m \leq 3$ then the linear system $\L(d;m^{10})$ is empty.
\end{corollary}

The following two propositions
allow us to prove the SHGH conjecture
in its original form \ref{conj:GHH} in the case $k\leq 9$. 

\begin{proposition}\label{prop:antiKnef}
Suppose we are in the anticanonical case, and 
$\L=\L(d;m_1p_1,...,m_kp_k)$ is an effective and nef system.
If either $\L\cdot K_{\bar \P}\leq -1$ or $p_1,...,p_k$ are general points
(hence $k\leq 9$ since we are in the anticanonical case),
then $\L$ is non--special.
\end{proposition}

\begin{proof}
This follows from Theorem III.1 from \cite {Har3}.
\end{proof}

\begin{proposition}\label{prop:leq9b}
Consider an effective system $\L=\L(d;m_1p_1,...,m_kp_k)$ with $k\leq 9$
and suppose we are in the strong anticanonical case. Then
$\L$ is not nef if and only if
there is a smooth rational curve $C$ on $\bar \P$
with $-2 \leq C^2 \leq -1$ and $\L\cdot C < 0$.
If $p_1,...,p_k$ are general points, then only the case $C^ 2 = -1$ is possible.
\end{proposition}

\begin{proof}
If $\L$ is not nef,
there is some irreducible curve $C$ such that $\L \cdot C < 0$
and  $C^2 < 0$.
On the other hand $C \cdot K_{\bar \P} \leq 0$.
By the adjunction formula,  $C$ is rational
and it is either a $(-1)$--curve or $C^2 = -2$.
However  this latter case cannot occur
if $p_1,...,p_k$ are general points (see \cite {cmsegre}, Corollary 5.4, or \cite {def}).\end{proof}

\begin{corollary}\label {cor:GHHkleq9} The SHGH conjecture holds for $k\leq 9$ 
general points.
\end{corollary}

\begin{proof}
Let $\L$ be an effective linear system.
If $\L$ is nef we are done by Proposition \ref {prop:antiKnef}.
If $\L$ is not nef,
then by Proposition \ref {prop:leq9b}
there are disjoint $(-1)$--curves $E_1,\ldots,E_h$
such that $\L\cdot E_i = -n_i < 0$, $i=1,\ldots,h$.
If $n_i\geq 2$ for some $i$, then $\L$ is $(-1)$--special. 
Otherwise the system $\L'=\L-(E_1+\cdots+E_h)$ is nef,
effective and has the same dimension and virtual dimension as $\L$.
We therefore finish by applying Proposition
\ref {prop:antiKnef} to $\L'$.
\end{proof}

\begin{remark}\label{rm:bcf}
{\rm 
Suppose $k\leq 9$ and the linear system $\L=\L(d;m_1,\ldots,m_k)$
is effective and nef.
Then it is fixed component free unless $k=9$
and $\L$ is a multiple of the anticanonical system.
In fact one has $\L\cdot K_{\bar \P}\leq 0$ and $\L^ 2\geq 0$.
If $k\leq 8$, then $\L\cdot K_{\bar \P}\leq -1$ by the index theorem,
and the assertion follows by Theorem III.1 from \cite {Har3}
(note that in case (b) of that reference,
there cannot be a fixed component
because of the generality of the base points).
If $k=9$, then the same argument works if $\L\cdot K_{\bar \P}\leq -1$.
If  $\L\cdot K_{\bar \P}= 0$ then $\L$ is a multiple of the anticanonical system.
}
\end{remark}

We close this section with a useful lemma.

\begin{lemma}\label{lem:us}
Consider an effective system $\L=\L(d;m_1,...,m_k)$ with $k\leq 9$.
If $E$ is a $(-1)$--curve such that $E\cdot \L = -n < 0$,
then $E$ appears in the base locus of $\L$ exactly with multiplicity $n$.
\end{lemma}

\begin{proof}
Let $E_1,...,E_h$ be all $(-1)$-curves such that $E_i\cdot \L = -n_i < 0$.
We may assume $E = E_1$.
Then $E_1,...,E_h$ is a $(-1)$--configuration in the sense of \cite{CM98},
i.e. $E_i \cdot E_j = 0$ if $1 \leq i < j \leq h$.
Moreover $\L = \sum_{i=1}^h n_iE_i + \L'$,
where $\L'=\L(d';m'_1,...,m'_k)$ is effective and not $(-1)$--special;
hence it is non--special, and $\L'\cdot E_i=0, i=1,...,h$. 

Then $\L'$ is nef
and by Remark \ref {rm:bcf} it has no $(-1)$--curve in its base locus.
\end{proof}

\section{The Nagata conjecture} \label{sec:nagata}

The SHGH conjecture implies another famous open
conjecture by Nagata (see \cite{nagata} and  \cite{cmsegre}):

\begin{conjecture}[Nagata] \label{conj:nagata}
The system $\L(d;m_1,...,m_k)$ is empty as soon as $k\geq 10$ and
\[
\sum_{i=1}^ k m_i \geq d\sqrt k.
\]
\end{conjecture}

Nagata's conjecture holds if $k$ is a perfect square (see \cite{nagata}).
In this case also SHGH conjecture holds 
(see \cite {evain}, \cite {ccnagata}, \cite {roe}).

For homogeneous linear systems, Nagata's conjecture reads as follows.

\begin{conjecture}[Homogeneous Nagata] \label{conj:nagatahom} Assume $k\geq 10$. 
The system $\L=\L(d;m^ k)$ is empty as soon as  $\L^ 2\leq 0$, i.e. as soon as
$$d/m\leq \sqrt k.$$
\end{conjecture}

For homogeneous linear systems with
non--positive self--intersection
Nagata's conjecture and the SHGH conjecture are equivalent: 
 in this case they both predict that the system 
is empty. Nagata's conjecture does not directly make any prediction for
homogeneous linear systems with positive self--intersection as SHGH does. However, as we are going to show next, it can be seen as
an asymptotic version of the SHGH conjecture. 

\begin{proposition}\label{prop:nagseg}  Fix a number $x\geq \sqrt {k}$ and suppose
that  if $\L(\delta;\mu^ k)$ is not empty, then $\delta/\mu> k/x$. 
Then for all pairs $(d,m)$ of positive mumbers 
such that $d/m\geq x $ the linear system $\L=\L(d;m^ k)$ is ample. 
Morever  there is an integer $N(d/m)$ such that for all
$n>N(d/m)$ the linear system $\L(nd;(nm)^ k)$ is non special.
\end{proposition}

\begin{proof} One has $\L^ 2>0$. This is clear
if $k$ is not a perfect square; if $k$ is a perfect square, it follows from Nagata's 
conjecture, which holds in this case.

Suppose $C$ is an irreducible curve on
$\bar \P$ such that $\L\cdot C\leq 0$. Since $\L^ 2>0$, there is a 
suitable multiple of $\L$ which is effective. This implies that $C^ 2\leq 0$.
Now uniformize $C$ in the sense of Nagata (see  \cite{nagata}, p. 285),
thus getting the homogeneous linear system $\L'=\L(\delta;\mu^ k)$ formed by
the sum of $C$ with its transforms via all
permutations of the base points $p_1,\dots,p_k$
which change $C$ into a different curve. One has $\delta/\mu>k/x$.
By monodromy, one also has $d\delta-km\mu=\L \cdot \L'\leq 0$,
which leads to a contradiction. 

The ampleness assertion follows by the Nakano--Moishezon Theorem.

As a consequence, there is a positive number $\epsilon$ such that all
linear systems of the form
$\L(\delta;\mu^ k)$ with $\delta/\mu\geq x-\epsilon$ are big and nef.  

Now consider all homogeneous linear systems of the form $n\L-K_{\bar \P}=\L(nd+3;(nm+1)^k)$. If $n$ is large enough, we have $(nd+3)/(nm+1)\geq x-\epsilon$.  Thus $n\L-K_{\bar \P}$ is big and nef and therefore $h^ 1(\bar \P, n\L)=0$
by the Mumford vanishing theorem \cite {M}. \end{proof}

\begin{corollary}\label {cor:nagseg}  If $k\geq 10$ and the homogeneous Nagata conjecture holds for $k$ 
general points, then the ray generated by $\L(d;m^ k)$ belongs to the effective cone
if and only if it belongs to the ample cone, in which case
there is an integer $N(d/m)$ such that for all
$n>N(d/m)$ the linear system $\L(nd;(nm)^ k)$ is non--special.

\end{corollary}

On the other hand, one has the following:

\begin {proposition}\label {amplecone} Assume $k\geq 10$ and
fix a number $x\geq \sqrt {k}$. Suppose
that  for all pairs $(d,m)$ of positive mumbers 
such that $d/m\geq x $, there is an integer $n$ such that
the linear system $\L(nd;(nm)^ k)$ is not empty and non--special.
Then all linear systems $\L(d;m^ k)$
with $d/m\geq x$ are big and nef.
\end{proposition}

\begin{proof} 
Set $\L=\L(d;m^ k)$. Replace $\L$ with a multiple so as to get an effective,
non--special linear system.
Assume $C$ is an irreducible curve such that $\L\cdot C< 0$.
Then $\L(-C)$ is effective and the exact sequence
$$0\to \L(-C)\to \L\to \L_{\vert C} \to 0$$
implies that $0=h^ 1(\bar \P,\L)=h^ 1(C, \L_{\vert C})$. This yields
$\deg( \L_{\vert C})\geq p_a(C)-1$, which forces $C$ to be 
a $(-1)$--curve (see \cite{def}). But since there is no
homogeneous $(-1)$--configuration when one blows up $k\geq 10$
points in the plane (see \cite {CM00}), we find a contradiction.
\end{proof}

\begin{remark} {\rm The above proof shows that the only
way a system $\L=\L(d;m^ k)$ can be not ample, is 
because of the existence of a curve $C$ with $p_a(C)=1$
and $\L\cdot C=0$.}
\end{remark}

As a consequence we have:

\begin{corollary}\label{cor:segrenagata} In the same hypotheses
as in Proposition \ref {amplecone}, if $\L(\delta;\mu^ k)$ is not empty, then
$\delta/\mu\geq k/x$. 

 In particular, if for all pairs $(d,m)$ of positive mumbers 
such that $d/m\geq \sqrt k $, there is an integer $n$  such that
 the linear system $\L(nd;(nm)^ k)$ is not empty and non--special, then
the homogeneous Nagata conjecture holds.
\end{corollary}

\begin{proof} Set $\L'=\L(\delta;\mu^ k)$.
Assume $\delta/\mu< k/x$. Then we can choose positive  integers
$d,m$ such that $x<d/m<k\mu/\delta$. Then the linear system $\L(d;m^ k)$ is 
big and nef, but $\L\cdot \L'<0$, a contradiction.  
\end{proof}

\section{The first degeneration}
\label{sec:firstdeg}

From now on, we will consider homogeneous linear systems
$\L(d;m^{10})$ with ten general base points.
In order to show that the linear system $\L(d;m^{10})$
has the expected dimension,
we will make an appropriate degeneration,
both of the plane (blown up at the ten general points)
and of the line bundle.
The full analysis will require several different degenerations,
depending on $d$, $m$, and their ratio $d/m$;
the first degeneration we will present has been described in \cite{CM98}
and \cite{CM00}.

We first consider the trival family $\Delta\times \P^ 2\to \Delta$
over a disc $\Delta$ and blow up a point in the central fiber. 
We thus get a flat and proper family $\Y\to \Delta$ over $\Delta$,
where the general fibre $Y_t$ for $t\neq 0$ is a $\P^2$,
and the central fibre $Y_0$ is reducible surface
$\P\cup\F$, where $\P\cong \P^ 2$ is a projective plane,
$\F\cong \F_1$ is a plane blown up at a point,
and $\P$ and $\F$ meet transversally along a smooth rational curve $E$
which is the exceptional divisor on $\F$ and a line on $\P$  (see Figure \ref {fig:1}).

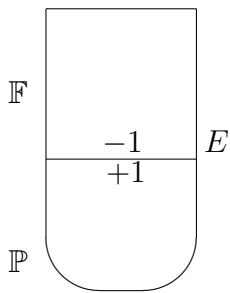
\begin{figure}[ht]
\setlength{\unitlength}{0.5mm}
\begin{center}

\begin{picture}(20,64)(15,15)

\put(0,40){\line(0,1){40}}
\put(0,40){\line(1,0){40}}
\put(40,40){\line(0,1){40}}
\put(0,80){\line(1,0){40}}
\put(20,40){\oval(40,70)[b]}
\put(-10,55){$\F$}
\put(-10,10){$\P$}
\put(16,34){$+1$}
\put(15,42){$-1\;\;\;\;\;\;\;E$}

\end{picture}
\end{center}
\caption{the degeneration of the plane}\label{fig:1}
\end{figure}

We now choose four general points on $\P$ and six general points on $\F$.
Consider these ten points as limits of ten general points in the general fibre $Y_t$
and simultaneously blow these points up in the family $\Y$.
This creates ten surfaces $R_i$, ruled over $\Delta$,
whose intersection with each fiber is a $(-1)$-curve,
the exceptional curve for the blow--up of that point in the family.
We denote by $\X_1\to \Delta$ this new family.
The general fibre $X_{1,t}$ for $t\neq 0$
is a plane blown up at ten general points.
The central fibre $X_{1,0}$, shown in Figure \ref {fig:2}, is the union $V_1 \cup Z_1$ where:

\begin{itemize}
\item $V_1$ is a plane blown up at four general points;
\item $Z_1$ is a plane blown up at seven general points;
\item $V_1$ and $Z_1$ meet transversally
along a smooth rational curve $E$
which is a $(-1)$-curve on $Z_1$, whereas $E^2=1$ on $V_1$: it is the pull--back of a line.
\end{itemize}

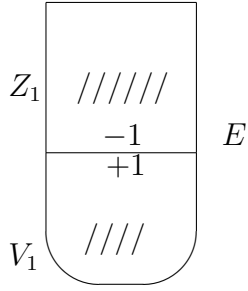
\begin{figure}[ht]
\setlength{\unitlength}{0.5mm}
\begin{center}
\begin{picture}(20,64)(15,15)

\put(0,40){\line(0,1){40}}
\put(0,40){\line(1,0){40}}
\put(40,40){\line(0,1){40}}
\put(0,80){\line(1,0){40}}
\put(20,40){\oval(40,70)[b]}
\put(-10,55){$Z_1$}
\put(-10,10){$V_1$}
\put(16,34){$+1$}
\put(15,42){$-1\;\;\;\;\;\;\;\;\;E$}

\put(10,15){$/ / / /$}
\put(8,55){$/ / / / / /$}

\end{picture}
\end{center}
\caption{the degeneration of the blown--up plane}\label{fig:2}
\end{figure}

Consider the line bundle $\L_0=\pi^*(\O_{\P^ 2} (d)) \otimes \O_{\X_1}(-\sum_i m R_i)$, 
where $\pi:  \X_1\to \P^ 2$ is the natural map. 
This restricts to $\L(d;m^{10})$ on the general fibre,
whereas on the central fibre it is
$\L(0;m^{4})$ on $V_1$ and $\L(d;0,m^{6})$ on $Z_1$. In this notation and in the 
following, the first multiplicity
for bundles on $Z_1$ refers to the point corresponding to $E$.

We will further twist this bundle by a suitable multiple of $Z_1$.
Namely,
we choose a parameter $a$ (to be determined later),
and define
\[
\L_1:= \L_0 \otimes \O_{\X}( (2m+a) Z_1).
\]
We will denote by $\L_{V_1}$ and $\L_{Z_1}$
the restrictions of $\L_1$ to $V_1$ and $Z_1$;
these bundles have the form
\[
\L_{V_1} = \L(2m+a; m^4), \quad \L_{Z_1} = \L(d; 2m+a, m^6).
\]

Using this degeneration we have
(compare with \cite{Har3}, Lemma II.7):

\begin{proposition}\label{prop:ratio>=10/3}
If $d/m \geq 10/3$ then $\L(d;m^{10})$ has the expected dimension.
\end{proposition}

\begin{proof}
We set $a=0$ for this analysis.
We will use  semi-continuity
and the transversality of the restricted linear systems
to the double curve $E$ (see \cite{CM98}, Proposition 3.3, (b)).
Note that these restricted systems both have degree $2m$.

First we notice that $\L_{V_1}$ and $\L_{Z_1}$
are both not empty and non--special.

As for $\L_{V_1} = \L(2m; m^4)$, this is 
the dimension $m$ linear system whose
elements are $m$ conics through the four points.  The dimension $r_{V_1}$
of the restriction of this system to $E$ is also $m$.

As for $\L_{Z_1} = \L(d; 2m, m^6)$, we apply Proposition \ref {prop:leq9b}
to prove  that it is nef.  The linear system $\vert -K_{Z_1}\vert$,
corresponding to $\L(3;1^ 7)$,
is base point free of dimension 2
and all $(-1)$--curves are contained in a curve of $\vert -K_{Z_1}\vert$. 
The crucial computation to make is
to intersect the bundle with the $(-1)$--curve which is the
proper transform of the unique cubic curve $C$ in $\L(3;2,1^ 6)$;
this intersection number is $3d-10m\geq 0$. By Proposition
\ref {prop:antiKnef} the system is non--special, and it is non--empty,
since the virtual dimension is positive
(which is also implied by the inequality $d/m \geq 10/3$).

Let  $r_{Z_1}$ be the dimension of the restriction of this system to $E$.
One has
$$r_{Z_1}=\dim( \L(d; 2m, m^6))-\dim(\L(d; 2m+1, m^6))-1$$
since $\L(d; 2m+1, m^6)$  is the \emph {kernel} subsystem of elements 
of $\L(d; 2m, m^6))$ that restrict to zero on $E$.
If $d/m > 10/3$, then, similar considerations as above imply that 
$\L(d; 2m+1, m^6)$ is not empty and non--special.
Therefore $r_{Z_1}$ can be computed as $2m$,
i.e. the restricted system is complete.
If $d/m =10/3$,
the cubic curve $C$ splits off twice from the subsystem $\L(d; 2m+1, m^6)$,
but arguments as above show that the residual system is non--special.
Hence $\L(d; 2m+1, m^6)$ has speciality exactly one, and $r_{Z_1}=2m-1$.
  
Since in either case $r_{V_1}+r_{Z_1}\geq 2m-1$,
we can apply Proposition 3.3, (b) from \cite{CM98} to conclude.
\end{proof} 

\section{Throwing (-1)-curves}\label{sec:throw}

The reason why a relatively simple degeneration
such as the one presented above
will not suffice to prove the general statement
is that the line bundles on the individual surfaces may become special.
By the SHGH Conjecture,
this should be a consequence of having $(-1)$-curves on those surfaces
intersecting the bundle negatively.

Our technique to handle this situation will be to blow up the
offending $(-1)$-curve, and twist by an appropriate multiple
of the exceptional ruled surface. We hope to arrive at the situation
where the ruled surface is a $\P^1 \times \P^1$, and can be blown
down via the other ruling, contracting the original $(-1)$-curve.
This process will create exceptional curves on the surfaces that
the $(-1)$-curve meets.
We refer to this technique in general as \emph{throwing} the $(-1)$-curve.
We explain this idea in two cases which will be relevant for
our purposes.

\subsection{\bf A 1-Throw.}
To be specific,
suppose that $E$ is a $(-1)$-curve
on one of the components $V$ of a local normal crossings semistable degeneration
and that the restriction $\L_V$ to $V$ of the line bundle $\L$ on the threefold
has the property that $\L_V\cdot E = -k < 0$.
Suppose further that there is a double curve $R$ where $V$ meets another component $Z$,
and $E$ meets transversally that curve $R$ at a single point $p$
and meets no other double curve.

Blow up the curve $E$, obtaining the ruled surface $S$; 
by the Triple Point Formula (cf.\ \cite {persson}, \cite{Frie})
$S$ will be isomorphic to $\P^1\times\P^1$,
and will meet $V$ along $E$,
which is one of the members of one of the rulings on $S$.
This blow--up will effect a blow--up of the other surface $Z$ at the point $p$,
creating a new $(-1)$-curve $E'$ on the blow--up $Z'$;
the surface $S$ also meets $Z'$ along $E'$,
which is a member of the other ruling on $S$ (see the left hand side of
Figure \ref {fig:3}).

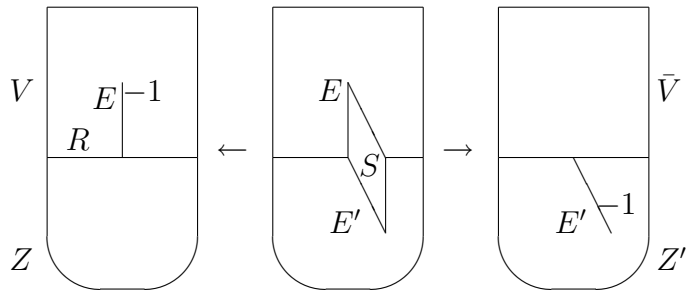
\begin{figure}[ht]
\setlength{\unitlength}{0.5mm}
\begin{center}
\begin{picture}(120,64)(20,20)
\put(0,40){\line(0,1){40}}
\put(0,40){\line(1,0){40}}
\put(40,40){\line(0,1){40}}
\put(0,80){\line(1,0){40}}
\put(20,40){\oval(40,70)[b]}
\put(20,40){\line(0,1){20}}
\put(-10,55){$V$}
\put(-10,10){$Z$}

\put(5,42){$R$}
\put(12,53){$E$}
\put(20,55){$-1$}

\put(45,40){$\leftarrow$}

\put(60,40){\line(0,1){40}}
\put(60,40){\line(1,0){20}}
\put(90,40){\line(1,0){10}}
\put(100,40){\line(0,1){40}}
\put(60,80){\line(1,0){40}}
\put(80,40){\oval(40,70)[b]}
\put(80,40){\line(0,1){20}}

\put(72,55){$E$}
\put(80,40){\line(1,-2){10}}
\put(80,60){\line(1,-2){10}}
\put(90,20){\line(0,1){20}}
\put(83,35){$S$}
\put(75,20){$E'$}

\put(105,40){$\rightarrow$}

\put(120,40){\line(0,1){40}}
\put(120,40){\line(1,0){40}}
\put(160,40){\line(0,1){40}}
\put(120,80){\line(1,0){40}}
\put(140,40){\oval(40,70)[b]}

\put(140,40){\line(1,-2){10}}
\put(135,20){$E'$}
\put(146,25){$-1$}
\put(162,55){$\bar{V}$}
\put(162,10){$Z'$}
\end{picture}
\end{center}
\caption{throwing a $(-1)$--curve}\label{fig:3}

\end{figure}

Note that the normal bundle of $S$ in the threefold has bidegree $(-1,-1)$,
and that the pullback of the bundle $\L$, restricted to $S$, has bidegree $(-k,0)$.

Replace the line bundle $\L$ by $\L' = \L \otimes \O(-kS)$.
The restrictions of $\L'$ to the various components are as follows:
$\L'|_V = \L_V(-kE)$;
$\L'|_{Z'} = \pi^*(\L_Z)(-kE')$;
$\L'|_S$ has bidegree $(0,k)$. Thus the new bundle on the surface $V$ does not meet $E$ anymore.

With this we see that we may blow $S$ down to $E'$ via the other ruling,
obtaining an alternate degeneration;
this will blow down the original $(-1)$-curve $E$ as desired,
and retain the blow--up $Z'$ of the surface $Z$.
The surface $V$ is blown down to $\bar{V}$,
and the bundle on $\bar{V}$ is simply the bundle on $V$, with $E$ removed.
On $Z'$ the bundle is the pullback of the original bundle on $Z$,
twisted by $-k$ times the exceptional divisor $E'$.
This corresponds to adding a point of multiplicity $k$ to the system on $Z$
(see the right hand side of Figure \ref{fig:3}).

This operation will be referred to as a \emph{$1$-throw} (of $E$ on $V$).

\subsection{\bf A 2-Throw.} 
Let us now consider the case when the $(-1)$-curve $E$
meets  transversally the double curve locus $R$ in two points $p_1$ and $p_2$.
We still assume that $E$ lies on the component $V$
and that the restricted system $\L_V$ has the property that $\L_V \cdot E = -k < 0$.
Again blow up $E$, obtaining the ruled surface $T$, which
is isomorphic to  $ \F_1$ by the Triple Point Formula;
$T$ meets $V$ along $E$, and this is also the $(-1)$-curve
which is the negative section of $T$.
The blow--up will create a blow--up $Z'$ of the surface (or surfaces) $Z$
that meet $V$ along $R$, at the two points $p_1$ and $p_2$, 
with two exceptional divisors $G_1$ and $G_2$ on $Z'$.
These $G_i$ are fibers of the ruling of $T$. This is shown on the left
side of Figure \ref {fig:4}. 

Now blow up $E$ again, creating the ruled surface $S$.
This time $S \cong \P^1\times\P^1$; $S$ meets $V$ along $E$,
and it meets $T$ along the negative section. 
The blow--up effects a further two blow--ups on $Z'$, creating the
surface $Z''$, and two more exceptional divisors $F_1$ and $F_2$
respectively, which are $(-1)$--curves on $Z''$. 
By abusing notation we denote by $G_1,G_2$ their proper transforms
on $Z''$; these are now $(-2)$--curves.
The surface $S$ now occurs with multiplicity two in the
central fiber of the degeneration,
since it was obtained by blowing up a double curve,
and its normal bundle in the total space 
 of the degeneration has bidegree $(-1,-1)$.
 All this is shown in the central part of Figure \ref {fig:4}. 

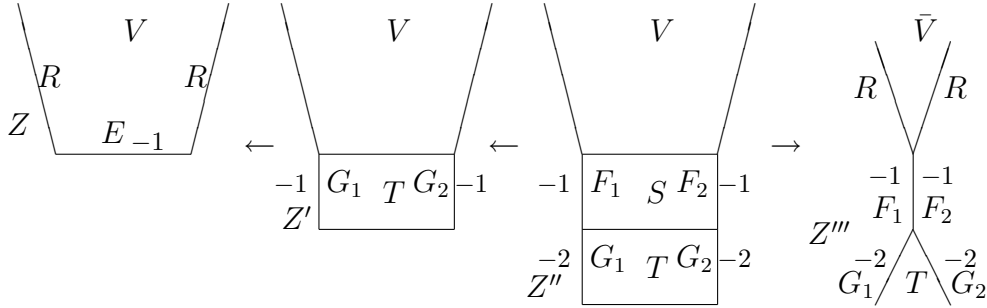
\begin{figure}[ht]

\setlength{\unitlength}{0.5mm}
\begin{center}
\begin{picture}(200,74)(20,10)
\put(0,40){\line(-1,4){10}}

\put(0,40){\line(1,0){36}}

\put(36,40){\line(1,4){10}}
\put(18,70){$V$}
\put(-13,45){$Z$}
\put(-5,58){$R$}
\put(34,58){$R$}
\put(12,43){$E$}
\put(20,41){\mbox{\footnotesize $-1$}}

\put(50,40){$\leftarrow$}

\put(70,40){\line(-1,4){10}}
\put(70,40){\line(1,0){36}}
\put(106,40){\line(1,4){10}}
\put(88,70){$V$}
\put(70,20){\line(0,1){20}}
\put(70,20){\line(1,0){36}}
\put(106,20){\line(0,1){20}}
\put(60,20){$Z'$}
\put(87,27){$T$}
\put(72,30){$G_1$}
\put(58,30){\mbox{\footnotesize $-1$}}
\put(95,30){$G_2$}
\put(106,30){\mbox{\footnotesize $-1$}}

\put(115,40){$\leftarrow$}

\put(140,40){\line(-1,4){10}}
\put(140,40){\line(1,0){36}}
\put(176,40){\line(1,4){10}}
\put(158,70){$V$}
\put(140,20){\line(0,1){20}}
\put(140,20){\line(1,0){36}}
\put(176,20){\line(0,1){20}}
\put(140,0){\line(0,1){20}}
\put(140,0){\line(1,0){36}}
\put(176,0){\line(0,1){20}}
\put(157,27){$S$}
\put(125,2){$Z''$}
\put(142,30){$F_1$}
\put(128,30){\mbox{\footnotesize $-1$}}
\put(165,30){$F_2$}
\put(176,30){\mbox{\footnotesize $-1$}}
\put(157,7){$T$}
\put(142,10){$G_1$}
\put(128,10){\mbox{\footnotesize $-2$}}
\put(165,10){$G_2$}
\put(176,10){\mbox{\footnotesize $-2$}}

\put(190,40){$\rightarrow$}

\put(228,40){\line(-1,3){10}}
\put(228,40){\line(1,3){10}}
\put(228,70){$\bar V$}
\put(228,20){\line(0,1){20}}
\put(228,20){\line(-1,-2){10}}
\put(228,20){\line(1,-2){10}}
\put(212,55){$R$}
\put(236,55){$R$}
\put(200,16){$Z'''$}
\put(217,23){$F_1$}
\put(230,23){$F_2$}
\put(208,3){$G_1$}
\put(238,3){$G_2$}
\put(226,3){$T$}
\put(216,32){\mbox{\footnotesize $-1$}}
\put(230,32){\mbox{\footnotesize $-1$}}
\put(212,10){\mbox{\footnotesize $-2$}}
\put(236,10){\mbox{\footnotesize $-2$}}

\end{picture}
\end{center}

\caption{throwing a $(-2)$--curve}\label{fig:4}

\end{figure}

Write $k = 2\ell-\epsilon$, with $\epsilon \in \{0,1\}$,
and replace the line bundle $\L$ by
$\L' = \L \otimes \O(-k S)\otimes \O(-\ell T)$.
The restrictions of $\L'$
to the various components are as follows:
$\L'|_V = \L_V(-kE)$;
$\L'|_{Z''} = \pi^*(\L_Z)(-k(F_1+F_2)-(\ell)(G_1+G_2))$;
$\L'|_T = \epsilon H$, where $H$ is the line class on $T \cong \F_1$;
$\L'|_S$ has bidegree $(0,\ell-\epsilon)$.

As in the case of the $1$-throw, we may now blow $S$ down the other way.
This contracts $E$ on the surface $V$, thus creating a new surface $\bar V$,
and contracts the negative section of $T$, so that $T$ becomes a $\P^2$.
The image $Z'''$ of the surface $Z''$
has the two curves $F_1$ and $F_2$ identified (see the rightmost side
of Figure \ref{fig:4}, where, by abusing notation, we still denote by $T$
its image after the contraction of $S$).

The bundle on the new plane created by $T$ has degree $\epsilon$.
The bundle on $Z'''$ can be interpreted in the geometry of $Z$
where two new \emph{compound multiple points}
have been created, each one a point of multiplicity $\ell$
and an infinitely near point of multiplicity $\ell-\epsilon$.
We will denote this type of compound multiple point
by the notation $[m_1,m_2]$, thus
indicating a multiple point $m_1$ and an infinitely near multiple point $m_2$.
Thus the above process produces two $[\ell,\ell-\epsilon]$ points on $Z$.

We refer to this operation as a \emph{$2$-throw} (of $E$ on $V$).

It is worth pointing out that
the contraction of $S$
which results in the identification of $F_1$ and $F_2$ on $Z''$
will force us to take this into account
when we will make the analysis of the linear systems on the degenerations.

Although one may imagine more complicated throws 
(when the $(-1)$--curve $E$ meets the double curve in more than two points)
we will not require such constructions in the sequel of this paper.

Note that in a $2$--throw,
if the two points $p_1$ and $p_2$ lie on the
same component of the double curve $R$,
then the surface $Z$ is a single component,
the curve $R$ becomes, after the $2$--throw, a nodal curve,
and the construction results in a non--normal component of the degeneration,
because of the identification of $F_1$ and $F_2$.
However this presents no real problems in the analysis;
the central fiber still has local normal crossings,
and all linear system computations on  the various components can be done
on their normalizations.

\section{Computation of the limit dimension}\label{sec:computation}

We will next perform a series of throws of $(-1)$--curves
starting from the first degeneration
described in section \S\ref{sec:firstdeg}.
This will create more complicated degenerations of the blown up plane,
which will have more than two components,
but still with local normal crossings and semistable.
These degenerations will carry a suitable limit of the relevant line bundle,
and it is our task to compute
the dimension of the space of sections of the limit bundle
to show that it is equal to the expected dimension.
Then, by appealing to semicontinuity,
we will prove non--speciality of the bundle on the general surface. 

Since we will have more than two components in the degeneration,
we cannot appeal directly to Proposition 3.3, (b) from \cite{CM98},
as we did in \S \ref {sec:firstdeg}, to make the computation of the dimension in the
limit. We therefore have to develop a more general analysis. 

In any event, as in the case of two components,
the space of limit sections is a fibre product,
namely, one must give sections on the components,
which agree on the double curves. 

In order to compute such a fibre product
it will be convenient for us to proceed iteratively,
by building up the degeneration one surface at a time.
This leads to an analogue of Proposition 3.3, (b) from \cite{CM98},
where the involved surfaces may be reducible.

To be specific,
suppose we have a (local normal crossings) surface $X_0= V\cup W$,
and a line bundle $\L$ on $X_0$,
restricting to $\L_V$ on $V$ and to $\L_W$ on $W$.
We denote by $C$ the intersection curve of $V$ and $W$,
with $\L$ restricting to $\L_C$ on $C$. 
Then, whether or not $V$ and $W$ are irreducible,
$H^ 0(X,\L)$ is the kernel of the difference map
\begin {equation}\label{eq:map}
H^ 0(V,\L_V)\oplus H^ 0(W,\L_W)\to H^ 0(C,\L_C).
\end{equation}
Geometrically this reads as follows: the curves in the linear system
$\L$ are Cartier divisors on $X_0$ and, as such, they have to be the union
of a curve in $\L_V$ and a curve in $\L_W$, which meet $C$ 
at the same points. 

If we know the dimension of the three spaces 
involved in the map \eqref {eq:map} and we also know that
the difference map is surjective, then we can compute the dimension
of $H^ 0(X,\L)$. The hypothesis of part (b) of Proposition 3.3, from \cite{CM98}
is equivalent to the surjectivity in that case. 

By considering the exact sequence (at the sheaf level)
\[
0 \to \L \to \L_V\oplus\L_W \to \L_C \to 0,
\]
we see that if $H^ 1(V,\L_V)=H^ 1(W,\L_W)=0$,
then $H^ 1(X_0,\L)=0$ if and only if
the difference map at the $H^0$ level is surjective.

In our case we will have $\pi:\X \to \Delta$
a flat, proper, semistable, local normal crossings degeneration 
of smooth projective surfaces $X_t$, $t\neq 0$, to the central fibre $X_0$; 
the total space $\X$ is endowed with a line bundle $\L_\X$, restricting to $\L$ on $X_0$.
This central fiber  is a divisor in the threefold $X$ of the form $\sum_{i=1}^ nV_i$.
We denote by $\L_i$ the restriction of $\L$ to $V_i$. 

Set $W_k=\sum_{i=1}^ kV_i$, and $\L^ {(k)}$ the restriction
of $\L$ to $W_k$; note that $W_k=W_{k-1}+V_k$.
Denote by $C_{k-1}$ the intersection of $W_{k-1}$ and $V_k$. The considerations
above apply and we can use them to compute $H^ 0(W_k,\L^ {(k)})$. For 
$k=n$ we have the desired space of sections. 

These arguments lead to the following statement:

\begin{proposition}\label {prop:fine} In the above setting, if:
\begin{itemize}
\item [{\rm (i)}] $H^1(V_i,{\L}_i)=0$, for all $i$;
\item [{\rm (ii)}] the difference maps 
$H^ 0(W_{k-1},\L^ {(k-1)})\oplus H^ 0(V_k,\L_k)\to H^ 0(C_{k-1},\L_{C_{k-1}})$
are surjective for all $k$;\end{itemize}
then $H^ 1(W_k,\L^ {(k)})=0$, for all $k$. Hence
$H^ 1(X_0,\L)=0$, and non--speciality of the bundle on the general 
surface $X_t$ follows by semicontinuity. 
\end{proposition}

\begin{remark}\label{rem:formal} {\rm
The surjectivity of the difference map in (ii) will 
follow, in our applications, exactly as in part (b) of Proposition 3.3 from \cite{CM98},
from a dimension count  and an appropriate transversality property, which will 
have to be checked case by case.
However a sufficient condition for the surjectivity
is that either one of the natural restriction maps is surjective,
which would follow from the vanishing of the appropriate $H^1$
(i.e., either $H^1(W_{k-1},\L^ {(k-1)}(-C_{k-1})) = 0$ or
$H^1(V_k,\L_k(-C_{k-1})) = 0$).
} 
\end{remark} 

\begin{remark}\label{rmk:fu}{\rm
We will have situations in which, due to the 
application of a $2$--throw,
a single component $V$ of $X_0$ is non--normal with a double curve $C$.
However, the normalization $\tilde V$ of $V$ will be smooth
and $V$ will be obtained by identifying
two non--intersecting curves $C_1$ and $C_2$,
both isomorphic to $C$.
Denote by $\L_{\tilde V}$, $\L_{C_i}$
the pull--backs of the bundle to $\tilde V$ and the $C_i$, respectively.
Then $\L_C$ injects into $\L_{C_1}\oplus \L_{C_2}$,
with quotient sheaf $\N$;
and $\L_V$ is the kernel of the natural map $\L_{\tilde V}\to \N$.
Therefore, if, as above, the corresponding map 
on global sections is surjective, and $H^ 1(\tilde V,\L_{\tilde V})=0$,
then $H^ 1(V,\L_{V})=0$. 
In particular 
$H^ 0(\tilde V,\L_{\tilde V}) \to  H^ 0(C,\N)$
is surjective if
$$H^ 0(\tilde V,\L_{\tilde V})\to H^ 0(C_1,\L_{C_1})\oplus H^ 0(C_1,\L_{C_1})$$
and
$$H^ 0(C_1,\L_{C_1})\oplus H^ 0(C_1,\L_{C_1})\to H^ 0(C,\N)$$
are both surjective. In our applications
the latter map will be surjective because $H^1(C,\L)=0$.

Alternatively, the above criterion can be deduced
from the cohomology of the exact sequence
$$0\to \L_V\to f_*\L_{\tilde V}\to \N\to 0,$$
and  $\N$, as above, is supported on $C$.

Again, from a geometric viewpoint, a curve in $\L_V$
corresponds to a curve in $\L_{\bar V}$ which meets
$C_1$ and $C_2$ in corresponding points, which are
glued on $V$. 

Note however that this is only a necessary condition:
in general, curves in $\L_{\bar V}$ meeting $C_1$ and $C_2$ 
in corresponding points
might not correspond to curves in $\L_V$. Actually, there could be 
more line bundles on $\bar V$ corresponding to the same line
bundle $\L_V$ on $V$. An easy example is the following: 
consider the curve of arithmetic genus 1 obtained by gluing
two distinct points $p_1,p_2$ on $\P^1 $. Then $V=\P^ 1\times C$ is
obtained from $\bar V=\P^ 1\times \P^ 1$ by gluing the two distinct
fibres $C_1,C_2$ over $p_1,p_2$. Take a non--trivial line bundle  
of degree 0 on $C$ and pull it back on $V$, thus getting a line 
bundle $\L_V$. Any such bundle corresponds to the trivial bundle
on $\bar V$, no non--zero section of which descends to a section
of $\L_V$.

In any event, $\dim{\L_V}$ is bounded above by the dimension
of the family of curves in $\L_{\bar V}$ meeting
$C_1$ and $C_2$ in corresponding points.}
\end{remark}

\section{The second degeneration: throwing the cubic}\label{sec:2deg}

In the proof of Proposition \ref{prop:ratio>=10/3},
the hypothesis $d/m \geq 10/3$ was used in a critical way
to show that the system on the surface $Z_1$ is nef,
and in particular to show that the intersection
with the cubic curve $C$ in the system $\L(3;2,1^6)$
is non--negative.
As soon as $d/m < 10/3$, this intersection becomes negative,
and we propose to employ a $2$-throw to remove it from $Z_1$,
creating a second degeneration.

We assume for this second degeneration that $16/5 \leq d/m < 10/3$.
Let us write 

\begin{equation}\label {eq:di}
d = 2c+e, \quad {\rm with}\quad e \in \{0,1\}.
\end{equation}

We return to the first degeneration,
and note that if $a \geq 0$, then
\[
C\cdot \L_{Z_1} =
3d-10m-2a
< 0.
\]

Hence $C$ splits exactly $10m-3d+2a$ times from $\L_{Z_1}$.
Furthermore it meets the double curve $E$ twice, at points $p_1$ and $p_2$.

We perform a $2$-throw of $C$ on $Z_1$,
blowing up $C$ twice and contracting the second ruled surface,
which is a $\P^1\times\P^1$, the other way.
Set
\begin{equation}\label {eq:bi}
b=5m+a-3c-e
\end{equation}
and note that $10m-3d+2a = 2b-e$,
so that the $2$-throw creates two $[b,b-e]$-points on $V_1$.
This results in our second degeneration, which now consists of three surfaces, as shown in Figure \ref {fig:5}: 

\begin{itemize}\label {item:2deg}
\item $V_2$, the transform of $V_1$.
The normalization $\tilde{V_2}$
is $V_1$ blown up at $p_1$ and $p_2$ (twice each),
creating exceptional curves $F_1$, $F_2$, $G_1$, and $G_2$.
We have $F_i^2 = -1$, $G_i^2 = -2$, $F_i \cdot G_j = \delta_{ij}$;
in addition, the $F_i$'s meet $E$ transversally (at the $p_i$'s).
The transform of the double curve $E$ now has self-intersection $-3$
on the normalization.
The surface $V_2$ is obtained from the normalization by suitably identifying $F_1$ and $F_2$.
The double curve $E$ becomes a nodal curve.
The linear system on the normalization $\tilde{V_2}$ has the form
$\L_{V_2} = \L(2m+a;m^4,[b,b-e]^2)$.

\item $Z_2$, the transform of $Z_1$. The surface
$Z_2$ is smooth: it is obtained from $Z_1$ by blowing down 
the curve $C$.
The linear system on $Z_2$ is obtained from the system $\L(d;2m+a,m^6)$ on $Z_1$
by removing $10m-3d+2a$ times the cubic (i.e. the system $\L(3;2,1^6)$),
and so has the form
$\L(10d-30m-6a;6d-18m-3a,(3d-9m-2a)^6)$.

\item $T_2$, the surface created by the $2$-throw, which is isomorphic
to a projective plane, meeting the surface $V_2$ along $G_1$ and $G_2$,
which are lines in $T_2$.  The linear system on $T$ has degree $e$.

\end{itemize}

\begin{figure}[ht]
\setlength{\unitlength}{0.5mm}

\setlength{\unitlength}{0.5mm}
\begin{center}
\begin{picture}(20,64)(20,20)
\qbezier(28,40)(0,75)(28,75)
\qbezier(28,40)(56,75)(28,75)
\put(26,62){$Z_2$}
\put(16,66){$E$}
\put(28,20){\line(0,1){20}}
\put(28,20){\line(-1,-2){10}}
\put(28,20){\line(1,-2){10}}
\put(00,16){$V_2$}
\put(17,23){$F_1$}
\put(30,23){$F_2$}
\put(08,3){$G_1$}
\put(38,3){$G_2$}
\put(26,3){$T_2$}
\put(16,32){\mbox{\footnotesize $-1$}}
\put(30,32){\mbox{\footnotesize $-1$}}
\put(12,10){\mbox{\footnotesize $-2$}}
\put(36,10){\mbox{\footnotesize $-2$}}

\end{picture}
\end{center}

\caption{the second degeneration}\label{fig:5}

\end{figure}
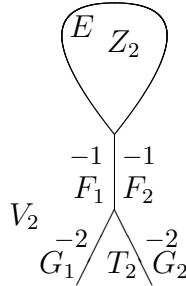

\medskip
Now let us analyze the linear systems on these three components.
First we note that the linear system on $T_2$
is non--special.

Let us turn our attention to $Z_2$, and set 
\begin{equation}\label{eq:alfa}
\alpha = d-3m.
\end{equation}
Hence we may write the system on $Z_2$ as
$\L(10\alpha-6a;6\alpha-3a,(3\alpha-2a)^6)$.
We have the series of quadratic transformations shown in the following table:
in each row 
the first number indicates the degree of the system,
the following numbers denote the multiplicities,
and we underline the base points used in each quadratic transformation:
\begin{equation}\label{Z2Cremona}
\begin{matrix}
10\alpha-6a; & \underline{6\alpha-3a}, 
& \underline {3\alpha-2a}, & \underline{3\alpha-2a},
& 3\alpha-2a, & 3\alpha-2a, 
& 3\alpha-2a,& 3\alpha-2a \\
8\alpha-5a; & \underline{4\alpha-2a},
& \alpha-a, & \alpha-a,
& \underline{3\alpha-2a}, & \underline{3\alpha-2a},
& 3\alpha-2a,& 3\alpha-2a \\
6\alpha-4a; & \underline{2\alpha-a},
& \alpha-a, & \alpha-a,
& \alpha-a, & \alpha-a,
& \underline{3\alpha-2a}, & \underline{3\alpha-2a} \\
4\alpha-3a; & 0 
& \alpha-a, & \alpha-a,
& \alpha-a, & \alpha-a,
& \alpha-a, & \alpha-a.
\end{matrix}
\end{equation}
  
The $0$ as the first multiplicity represents the cubic curve,
now blown down.
Hence the system on $Z_2$ is Cremona equivalent
to $\L(4\alpha-3a;(\alpha-a)^6)$,
which if $0 \leq a \leq \alpha$ is excellent
and therefore non--special by Proposition \ref {prop:strongantiK}. 

We next consider the system on $V_2$,
or rather on its normalization $\tilde{V_2}$.

\begin{lemma}\label{lem:V2} Assume $16/5 \leq d/m < 10/3$.
If
$5m-3c-e \leq a \leq d-3m$, i.e., if
$$b/2\leq a\leq \alpha,$$
then the system $\L_{V_2}$ is non--empty and non--special.
\end{lemma}

\begin{proof}
Recall that the linear system $\L_{V_2}$ 
has the form $\L(2m+a;m^4,[b,b-e]^2)$
on the normalization $\tilde{V_2}$.

We first claim that, with $a$ in the given range,
$\L_{V_2}$ is effective because its virtual dimension $v$ is non--negative.
After substituting $b=5m+a-3c-e$, $v$ becomes a function of $a$, $m$, $c$, and $e$;
for fixed $m$, $c$, and $e$,
one has $\partial v/\partial a = 6\alpha-3a-1/2 > 0$ in the given range.
Hence to see that $v \geq 0$,
it suffices to check this for the left endpoint $a=5m-3c-e$.
In this case
\[
2v=(5m-3c)(45c-71m)+(15c-23m)+6e(31m-19c-3).
\]
If $e=0$, then $d=2c$, and the inequalities on $d/m$ imply that
$8m\leq 5c$ and $3c < 5m$.
Hence $5m-3c \geq 1$, $45c-71m \geq m$, and $15c -23m \geq m$, so that $v \geq m$.
If $e=1$, then $d=2c+1$,
and the inequalities are $16m \leq 10c+5$ and $6c+3 < 10m$;
the first cannot be an equality for parity reasons, so that in fact $8m \leq 5c+2$,
and the second gives $3c+2 \leq 5m$.
Hence $5m-3c \geq 2$, and $45c-71m \geq m-18$,
so the quadratic part is at least $2m-36$.
The linear term is now $163m-99c-18$, and since $5m-3c \geq 2$,
we have $165m-99c \geq 66$; hence the linear term is at least $48-2m$.
Therefore $2v \geq (2m-36)+(48-2m) = 12$.

Note that we are in the anticanonical case, and 
the anticanonical pencil $\L(3; 1^4, [1,1]^2)$ of $\tilde{V_2}$
has the $(-3)$--curve $E$ as a fixed component,
and the movable part is the pencil $\L(2;1^ 4)$.

We have $b \leq 2m/5 < m$ since
$a \leq d-3m$ and $d/m\geq {16}/5$.
One has
$\L_{V_2} \cdot K_{\tilde{V_2}} =
18m-6d+a \leq -m$. 
By Proposition \ref{prop:antiKnef},
it suffices to prove that $\L_{V_2}$ is nef.
Let $D$ be a irreducible curve on $\tilde{V_2}$
such that $\L_{V_2}\cdot D < 0$.
Then the proof of Proposition \ref{prop:leq9b}
shows that $D$ can either be a $(-1)$--curve or a $(-2)$--curve
or the $(-3)$--curve $E$.
However $E$ is the curve in the system $\L(1;0^4,[1,1]^2)$
and so
$\L_{V_2}\cdot E =
6d - 18m -3a \geq 3m/5>0$.
By the structure of the anticanonical pencil,
the only $(-2)$--curves on $\tilde{V_2}$
are $G_1,G_2$, and we have
$\L_{V_2} \cdot G_i=e$ for $i=1,2$. 
Hence we are reduced to considering the $(-1)$--curves.

The antibicanonical system $\L(6; 2^4, [2,2]^2)$
again has $E$ as the fixed part
and the movable part is the $4$--dimensional system 
$\L(5; 2^4, [1,1]^2)$.
Since every $(-1)$--curve
is contained in a curve of the antibicanonical system,
we see that the $(-1)$--curves have degree at most $5$.

One can take these up in turn, by degree.
Those of degree $1$ are lines through two of the points;
they could be in one of the two systems
$\L(1;1^2,0^2,[0,0]^2)$ or $\L(1;1,0^3,[1,0],[0,0])$.
Since $m \geq b$,
the first type has the smallest intersection with $\L_{V_2}$;
this intersection is $a$, which is positive.

Those of degree $2$ are conics through five of the points;
the two with smallest intersection with $\L_{V_2}$ are in the systems
$\L(2;1^4,[1,0],[0,0])$ and $\L(2;1^4,[0,0],[1,0])$.
They meet the system with intersection number
 $a-(5m-3c-e)$ which is non--negative by hypothesis.

Those of degree $3$ which have smallest intersection with $\L_{V_2}$
are in the linear system $\L(3; 2, 1^3, [1,1], [1,0])$
(up to permutations of the points);
the  intersection number $(9d-28m-e)/2$ is positive since
$d/m\geq 16/5$.

Finally the curve of degree $4$ of interest is in the system
$\L(4; 2^3, 1, [1,1]^2)$;
this curve has intersection number $6d - 19m$ which
is positive since $d/m\geq 16/5$.

Since there are no $(-1)$-curves of degree $5$,
we have shown that $\L$ is nef, and this completes the proof.
\end{proof}

We may find an $a$ satisfying the hypothesis of the previous Lemma
if and only if $5m-3c-e \leq d-3m$;
this is equivalent to
   $5d-16m \geq e$, which is true if and only if $d/m \geq 16/5$.
This shows  that, under the assumptions of this section, condition (i) of Proposition \ref {prop:fine}, namely the 
non--speciality of the bundles on each component of the degeneration, holds with $a$ in the range
of Lemma \ref {lem:V2}.

As for  condition (ii)  of Proposition \ref {prop:fine}, we set $W_1=V_2$, $W_2=V_2+T_2$, $W_3=X_0$.

In order to compute $H^ 0(W_1,\L^ {(1)})$
we proceed as in Remark \ref {rmk:fu}
since $W_1=V_2$ is non--normal,
and is obtained from the normalization $\tilde V_2$ 
by identifying the two disjoint $(-1)$--curves $F_1$ and $F_2$. 
Applying those arguments,
we see that $H^1(W_1,\L^ {(1)}) = 0$
will follow from showing that
$H^0(\tilde V_2,\L_{\tilde V_2})$ surjects onto
the appropriate quotient of $H^0(F_1,\L_{F_1}) \oplus H^0(F_2,\L_{F_2})$;
it therefore suffices to show that the $H^0$ on $\tilde V_2$
surjects onto the direct sum itself,
and this is implied by the following:

\begin{lemma}\label {lem:V2F}
With the same numerical hypotheses as in Lemma \ref{lem:V2},
one has
$H^1(\tilde V_2,\L_{\tilde V_2}(-F_1-F_2))=0$.
\end{lemma}

\begin{proof}
The argument parallels the proof in Lemma \ref {lem:V2},
and therefore we will be brief.
We first show effectivity, and as above it suffices to prove
this for $a=5m-3c-e$; in this case 
we compute twice the virtual dimension $v$ to be now
\[
2v=(5m-3c)(45c-71m)+(39c-63m)+6e(31m-19c-1) - 4.
\]
If $e=0$, the quadratic part is again at least $m$,
and the linear part is at least $m-c$;
since $2m-c > 1$, we have $v \geq 0$.
If $e=1$, 
the quadratic part is (as above) at least $2m-36$,
and now the linear part is bounded below by $50-2m$
so that $2v \geq 4$.

Therefore the bundle is effective.
Moreover the intersection with the canonical bundle stays negative.
Finally one checks that the only curves
that may have negative intersection with the bundle
are some of the conics,
which may have intersection number $-1$ at least;
this does not cause speciality.
\end{proof}

Next we proceed to add the surface $T_2$,
creating $W_2$.
In this case the double curve is 
the union of two lines in the plane  $T_2$,,
which are glued to $G_1$ and $G_2$ on $V_2$.
The bundle on $T_2$, has degree $e$
and the restriction from the plane is surjective.
This proves that
$H^ 1(W_1,\L^ {(2)})=0$. 

Finally we add $Z_2$, completing the central fibre.
This time the double curve is the nodal curve $E$. 
We claim that the restriction map
from the sections  of the bundle
on $Z_2$ to the sections on $E$ is surjective,
which is sufficient for the criterion of Proposition \ref{prop:fine};
this then will give us the desired non--speciality of the bundle on $X_0$.
In fact, as we saw, the bundle on $Z_2$
is Cremona equivalent to $\L(4\alpha-3a;(\alpha-a)^6)$
and $E$   turns out to be Cremona equivalent to
a nodal curve in the linear system $\L(3;1^6)$.
The difference linear system is $\L(4\alpha-3a-3;(\alpha-a-1)^6)$.
This system is still excellent
and therefore non--special by Proposition \ref {prop:strongantiK}.
This implies the surjectivity of the restriction map.

We have now checked all the necessary details to conclude the following:

\begin{corollary}\label{16/5<=ratio}
If $16/5 \leq d/m < 10/3$
then the system $\L(d;m^{10})$ has the expected dimension.
\end{corollary}

\section{The third degeneration: throwing the conics}
\label{sec:3}

In what follows we assume that $19/6 \leq d/m < 16/5$.
Write $d$,
 $b$ and $\alpha$ as in \eqref {eq:di}, \eqref {eq:bi} and \eqref {eq:alfa}.
In this section we will assume the inequality $b > 2a$;
this is equivalent to $a < 5m-3c-e$.  We will also often assume $a\leq \alpha+1$.

We consider the second degeneration,
and we additionally twist the line bundle on the threefold by
$-(b-2a-e)T_2 = -(5m-a-3c-2e)T_2$.

We then have the configuration of the second degeneration,
but the bundles on $V_2$ and $T_2$ have now changed.
Since $T_2$ does not meet $Z_2$, the bundle on $Z_2$ is unchanged;
it is still of the form $\L(10\alpha-6a;6\alpha-3a,(3\alpha-2a)^ 6)$, which,
as we saw in \eqref{Z2Cremona}, is
Cremona equivalent to $\L(4\alpha-3a;(\alpha-a)^ 6)$.

The bundle on the plane $T_2$ now has degree
 $2b-4a-e>0$.

The bundle on $V_2$, which used to pull--back on $\tilde V_2$ as
$\L(2m+a;m^4,[b,b-e]^2)$,
has been twisted by
$-(b-2a-e)(G_1+G_2)$,
and now pulls back to  
$\L(2m+a; m^4, [2b-2a-e,2a]^2)$.

Let us consider the proper transforms on $\tilde{V_2}$ of the  two conics $C_1$, $C_2$ 
which are in the systems
\[
\L(2;1^4, [1,0], [0,0]), \quad \L(2;1^4, [0,0],[1, 0]).
\]
Their intersection number with the system on $\tilde{V_2}$ is
$ - (2(b-2a)-e)$, which is negative.

In this third degeneration we now execute two $2$-throws, one for each $C_i$.
Each $C_i$ will be removed $2b-4a-e$ times from the system on $V_2$,
and then blown down.
Each $C_i$ meets the double curve twice, but on different components:
$C_i$ meets $G_i$ transversally, and also meets the curve $E$.
Therefore we will create four new $[b-2a,b-2a-e]$-points,
two on the blow--up of $Z_2$
(let us call that surface $Z_3$ now),
and two on the blow--up of $T_2$
(which we call $T_3$ now).  The surfaces
$Z_3$ and $T_3$ will now meet along two curves $A_1, A_2$,
the exceptional divisors of the second blow--ups at the two points.

We also create two new planes $U_{1,3}$ and $U_{2,3}$
from the $2$-throw construction with $C_1$ and $C_2$,
respectively.
They will meet both $T_3$ and $Z_3$ along lines,
at the first blow--ups of each infinitely near point.
The bundle, restricted to $U_{i,3}$, has degree $e$.

The surface $V_2$ has both conics blown down;
we call the result $V_3$ (and its normalization $\tilde{V_3}$).
The linear system on $\tilde{V_3}$ is that of $\tilde{V_2}$,
with the two conics removed $ 2b -4a -e$  times.
It will be convenient to set 
$$\mu = 6d-19m$$
which is non--negative. Then linear system on $\tilde{V_3}$
corresponds to the system
$\L(9a+2\mu;(4a+\mu)^4,[2a,2a]^2)$.

Thus,  this third degeneration consists of the five surfaces
$V_3$, $Z_3$, $T_3$, $U_{1,3}$, and $U_{2,3}$, with double curves
as shown in  Figure \ref{fig:6}.

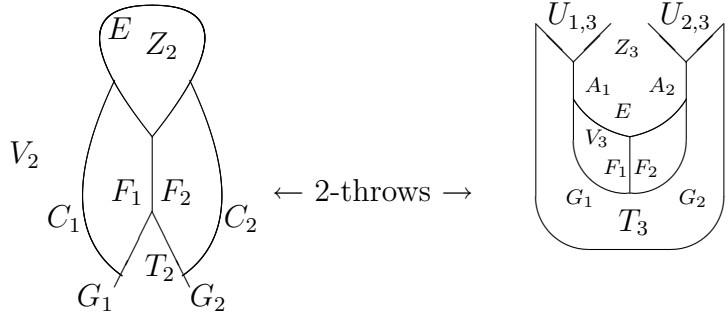
\begin{figure}[ht]
\setlength{\unitlength}{0.5mm}
\begin{center}
\begin{picture}(180,74)(10,5)

\qbezier(28,40)(0,75)(28,75)
\qbezier(28,40)(56,75)(28,75)
\qbezier(20,3)(0,17)(18,55)
\qbezier(36,3)(56,17)(38,55)
\put(26,62){$Z_2$}
\put(16,66){$E$}
\put(28,20){\line(0,1){20}}
\put(28,20){\line(-1,-2){10}}
\put(28,20){\line(1,-2){10}}
\put(00,16){$C_1$}
\put(47,16){$C_2$}
\put(17,23){$F_1$}
\put(30,23){$F_2$}
\put(08,-5){$G_1$}
\put(38,-5){$G_2$}
\put(-10,33){$V_2$}
\put(26,3){$T_2$}
\put(60,23){$\leftarrow$ $2$-throws $\rightarrow$}

\put(130,70){\line(1,-1){10}}
\put(140,60){\line(1, 1){10}}
\put(140,60){\line(0,-1){10}}
\qbezier(140,50)(145,42)(155,40)
\put(180,70){\line(-1,-1){10}}
\put(170,60){\line(-1, 1){10}}
\put(170,60){\line(0,-1){10}}
\qbezier(170,50)(165,42)(155,40)
\put(155,40){\line(0,-1){15}}
\put(155,50){\oval(30,50)[b]}
\put(155,70){\oval(50,120)[b]}
\put(133,70){$U_{1,3}$}
\put(163,70){$U_{2,3}$}
\put(148,62){ {\scriptsize $Z_3$} }
\put(148,45){ {\scriptsize $E$} }
\put(143,52) {{\scriptsize $A_1$} }
\put(160,52) {{\scriptsize $A_2$} }
\put(152,15){$T_3$}
\put(143,38){{\scriptsize $V_3$}}
\put(138,22){{\scriptsize $G_1$}}
\put(168,22){{\scriptsize $G_2$}}
\put(148,30){{\scriptsize $F_1$}}
\put(156,30){{\scriptsize $F_2$}}

\end{picture}
\end{center}

\caption{throwing the conics $C_1, C_2$}\label{fig:6}

\end{figure}

\begin{lemma}\label{lem:V3}
If $d/m \geq 19/6$ and $a\geq 0$
then the system $\L_{\tilde V_3}$ is non--empty and non--special.
\end{lemma}

\begin{proof}
The system on $\tilde V_3$ is not standard, and we have
the following series of quadratic transformations:
\[
\begin{matrix}
9a+2\mu; & \underline{4a+\mu}, & \underline {4a+\mu},
& \underline{4a+\mu}, & 4a+\mu, 
& [2a,2a], & [2a,2a] \cr
6a+\mu; & a, & a,
& a, & \underline{4a+\mu},
& [\underline{2a}, {2a}], & [\underline{2a},2a] \cr
4a+\mu; & a, & a,
& a, & \underline{2a+\mu},
& \underline{2a}, & \underline{2a} \cr
2a+\mu; & a, & a,
& a & \mu & & & \cr
\end{matrix}
\]

If $a\leq \mu$ the final system is excellent. Otherwise,
perform a quadratic transformation based at the three points of multiplicity
$a$, getting the system $\L(3a;\mu^ 4)$, which is excellent.
\end{proof}

The bundle $\L_{Z_3}$ on $Z_3$  is Cremona equivalent to
$\L(4\alpha-3a;(\alpha-a)^6, [b-2a,b-2a-e]^2)$,
where the two compound multiple points
lie on an irreducible nodal cubic curve
passing through the six points of multiplicity $\alpha-a$;
this cubic is the image of the original double curve $E$
under the Cremona transformation. The existence of $E$
implies that we are in the strong anticanonical case.

\begin{lemma}\label{lem:Z3}
If $19/6 \leq d/m < 16/5$ and $0\leq a\leq\alpha+1$
then the system $\L_{Z_3}$ is non--empty and non--special.
\end{lemma}

\begin{proof}
We have
\begin{equation}\label{eq:kappa}
\L_{Z_3} \cdot K_{Z_3} =-2\mu-a
\end{equation}
which is non--positive.
Moreover $b-2a >\alpha-a$, since this
is equivalent to $16m \geq 5d+e$.
Therefore the three largest multiplicities of the system
$\L(4\alpha-3a;(\alpha-a)^6, [b-2a,b-2a-e]^2)$
are those of the compound points.
Hence, after removing six innocuous $(-1)$--curves
in case $a=\alpha+1$, this system is standard
if $3b-e-3a\leq 4\alpha$,
which is equivalent to
$d/m \geq 54/17$. Note that $19/6 < 54/17 < 16/5$.
Thus if $d/m \geq 54/17$, we also have $\mu>0$ and the system is excellent;
therefore by part (iii) of Proposition \ref{prop:strongantiK}, 
the system is non--special.

If $d/m < 54/17$,
we perform a quadratic transformation based at the three points
of largest multiplicity (the $b-2a$, $b-2a$, and $b-2a-e$ points),
obtaining the system
$\L(25c+12e-39m-3a; (\alpha-a)^6, 7d-22m-a, 7d-22m-a-e, 7d-22m-a, 5m-3c-a-2e)$.
Now the three largest multiplicities are
$7d-22m-a$, $7d-22m-a$, and $5m-3c-a-2e$
(since $d/m \geq 19/6$);
their sum is equal to the degree $25c+12e-39m-3a$,
and so this system is standard.
We still have the intersection with the canonical class given by \eqref {eq:kappa} and so,
if either $d/m > 19/6$ or $a>0$, this system is excellent;
we conclude by applying  part (iii) of
Proposition \ref{prop:strongantiK}
as above.

If $d/m = 19/6$ and $a=0$, then the system is only almost excellent.
Its restriction to the anticanonical curve $E$ has degree $0$,
and, because of the generality of the original choice of the points,
the restriction to $E$ is a non--trivial bundle, which therefore
has no $H^ 1$, since $E$ has arithmetic genus $1$. The
kernel of the restriction to $E$ is excellent,
and therefore is non--special by  part (iii) of Proposition \ref{prop:strongantiK}.
The usual restriction exact sequence now shows that the original system {$\L_{Z_3}$}
is non--special as well. \end{proof}

Let us now turn our attention to the surface $T_3$,
the quadruple blow--up of the plane $T_2$;
the linear system there is 
$\L_{T_3} = \L(2b-4a-e; [b-2a,b-2a-e]^2)$.

\begin{lemma}\label{lem:T3} 
The system $\L_{T_3}$ is non--empty and non--special.
\end{lemma}

\begin{proof}  The  linear system $ \L_{T_3}$ is
composed of $b-2a-e$ conics in a pencil of conics bitangent to 
the pair of lines corresponding to $G_1$ and $G_2$,
plus, if $e=1$, the fixed line through the base points.  This system 
is clearly non--special.
\end{proof}

Finally the surfaces $U_{i,3}$ are planes, with systems
of non--negative degree on them, hence also non--special.

We now turn our attention to check the required surjectivity criteria 
from Proposition \ref {prop:fine}. In this case we set
\[
W_1=V_3, W_2=W_1+T_3, W_3=W_2+Z_3, W_4=X_0.
\]

The analysis is similar to that for the second degeneration.
We will need the following two lemmas.

\begin{lemma}\label {lem:V3F}
If $19/6 < d/m < 16/5$, and $a\geq 0$ then 
$H^1(\tilde V_3,\L_{\tilde V_3}(-F_1-F_2))=0$. 
\end{lemma}

\begin{proof}
As we saw in the proof of Lemma \ref {lem:V3},
the system $\L_{\tilde V_3}$ corresponds to 
$\L(9a+2\mu;
(4a+\mu)^ 4, [2a,2a]^ 2)$. By subtracting $F_1$ and $F_2$, 
one sees that the corresponding system is  $\L(9a+2\mu;
(4a+\mu)^ 4, [2a+1,2a]^ 2)$. If $a=0$, this is the system
$\L(2\mu; \mu^ 4, 1^ 2)$ which is non--special since
$\mu>0$. Assume $a>0$. Proceeding as  in the proof of Lemma \ref  {lem:V3}, we see 
that two $(-1)$--curves split once and the residual system is
Cremona equivalent to $\L(2a+\mu-2;a^ 3,\mu-2)$.

If $\mu=1$
one more $(-1)$--curve splits once and the residual system is 
$\L(2a-1;a^ 3)$. If $a=1$, this is empty with no $H^ 1$. If $a>1$, 
three more $(-1)$--curves split once and the residual system is 
$\L(2a-4;(a-2)^ 3)$, which is non--special.

If $\mu>1$, and $a\leq \mu-2$, the system is excellent,
hence non--special. Otherwise, perform a Cremona transformation
at the three points of multiplicity $a$, thus obtaining the system
$\L(3a;(\mu-2)^ 4)$, which is excellent, hence non special.
\end{proof}

Recall that the double curve created by adding $Z_3$ to $W_2$ is $E+A_1+A_2$,
where the $A_i$'s are the exceptional divisors
of the second blow--ups of the two compound multiple points on $Z_3$.
(As shown in Figure \ref {fig:6}, these are the intersection curves of $Z_3$ with $T_3$.)

\begin{lemma}\label {lem:Z3F} 
If $19/6 < d/m < 16/5$, and  $0\leq a\leq \alpha+1$
then 
$H^1(Z_3,\L_{Z_3}(-E-A_1-A_2))=0$.

\end{lemma}

\begin {proof} The proof parallels the one of Lemma \ref {lem:Z3}.

Proceeding as in \eqref {Z2Cremona}, we see that the system
in question corresponds to
$\L(4\alpha-3a-3;(\alpha-a-1)^6,[b-2a-e,b-2a-1]^ 2)$,
whose intersection with the canonical bundle is
$1-2\mu-5a$.

One has $b-2a-e>\alpha-a-1$,
since this is equivalent to $16m-5d-e\geq 0$. 
If $3b-2e-1-6a\leq 4\alpha-3a-3$, which
is equivalent to $17c+9e-27m\geq 2$,  the system is
excellent.  Otherwise perform a quadratic
transformation based at the points of largest multiplicities, i.e.
$b-2a-e, b-2a-e, b-2a-1$. 
The resulting system is 
$\L(25c+13e-39m-3a;(\alpha-a-1)^ 6, 14+7e-22m-a-2,14c+8e-22m-a-3,
14c+7e-22m-2, 5m-3c-a-e-1)$. 
Since $\mu>0$, the largest multiplicities are
$14c+7e-22m-a-2$ and $5m-3c-a-e-1$ and
the system is excellent. \end{proof}

Our next non--speciality result  is as follows.

\begin{proposition}\label{19/6<=ratio}
If $19/6 < d/m < 16/5$ 
then the system $\L(d;m^{10})$
has the expected dimension.
\end{proposition}

\begin{proof} 
The non--speciality of the system on $W_1$ 
follows from Lemma \ref  {lem:V3F} for any $a\geq 0$.

As for the glueing of $T_3$ to create $W_2$,
we note that the double curve is $G=G_1+G_2$,
which is the proper transform on $T_3$ of two lines in the plane $T_2$.
The bundle on $T_3$ restricts to $G$ as the trivial bundle,
and surjectivity follows by remarking that
$G$ is not in the base locus of the linear system on $T_3$ (see the proof of Lemma
\ref {lem:T3}). 

If $0\leq a\leq \alpha+1$, Lemma \ref {lem:Z3} implies that
we have non--speciality on $W_3$. Finally attaching the two planes $U_{1,3}, U_{2,3}$
does not create any speciality, since the bundles there 
have degree $e$, and we have non--speciality on the central fibre $W_4$.
\end{proof}

\section{Generality and transversality}
\label{sect:gentrasv}

In this section we focus on the generality of our choices in the above constructions
in order to prove some transversality properties needed in the sequel.

\subsection {\bf Correspondence generality.} \label {subsec:one}
First, let us go back to the second degeneration. 
Consider the two curves $F_1$ and $F_2$ on $V_1$ which are identified via a projective
transformation $\omega: F_1\to F_2$  in order to get $V_2$. Let $p_i,q_i$ the intersection points of $F_1$ with $E$ and $G_i$, $i=1,2$, respectively. One has $\omega(p_1)=p_2$ and $\omega(q_1)=q_2$. 
Our first claim is that  $\omega$ can be assumed to be general, given this constraint. This gives a one dimensional family of such projective transformations, depending on a parameter varying in $\mathbb C^ *$. 

To see this, go back to the first degeneration. After we choose the six general points on $\F$, we
have the cubic curve $C$ on $\F$, which we will throw in the second degeneration. This curve cuts $E$ in two points $x_1,x_2$, which will be blown up twice on $\P$ in the second degeneration creating 
the exceptional curves $F_1,G_1$ and $F_2,G_2$. 
Consider the projective transformations of $\P$ fixing $x_1$ and $x_2$. They form a group $\Omega$ of dimension
$4$. This group acts also on the double blow up of  $\P$ at $x_1$ and $x_2$, and therefore it acts on all
the curves $F_1,G_1$ and $F_2,G_2$. One sees that $\Omega$ induces on $F_1, F_2$ the full group of pairs of projective transformations fixing $p_1, q_1$ and $p_2,q_2$. We leave the easy proof to the reader.

Note now that we can can act by $\Omega$ on our choices of the remaining four general points on $\P$, that we blow up creating $V_1$. This implies our claim about the generality of $\omega$. 

The same considerations work also for the two curves
$A_1,A_2$  on $Z_3$ in the third degeneration, and the 
projective transformation between them  induced by the 
pencil of conics on the plane $T_2$ bitangent to the pair
of lines coresponding to $G_1$ and $G_2$.

\subsection{\bf Configuration generality.} \label {subsec:two} Next we go back to the third degeneration. 

As we saw in the proof of Lemma \ref {lem:V3}, the system $\L_{\tilde V_3}$ is Cremona equivalent to the system
$\L(2a+\mu;a^3 ,\mu)$. The reader may verify that, under the series of of quadratic transformations
performed in the proof of Lemma \ref {lem:V3}, the two curves $F_1,F_2$ map to two lines $L_1, L_2$ passing through the point of multiplicity $\mu$. This point arises as the contraction of the quartic curve
$\L(4;2,2,2, 1, [1,1]^ 2)$.  
The remaining three base points $y_1,y_2,y_3$ of multiplicity $a$ are general, and the lines joining them correspond to the quartic curves $\L(4;2,2 ,1,2, [1,1]^ 2)$, $\L(4;2,1,2,2, [1,1]^ 2)$, $\L(4;1,2,2,2, [1,1]^ 2)$.  

If we perform a further quadratic transformation based at $y_1,y_2,y_3$, then $L_1,L_2$ are mapped to two conics $\Gamma_1,\Gamma_2$, intersecting at four points
$a_1,\dots,a_4$ which arise as the contractions of all the aforementioned quartic curves.  The above generality considerations, tell us that $a_1,\dots,a_4$ correspond to general points of $F_1,F_2$.

In the resulting final Cremona  transformation, the curve $E$ maps to a line $L$  meeting $\Gamma_i$ at two points $u_i,v_i$, where $u_i$ is the image of the point $p_i$ intersection of $F_i$ with $E$, and $v_i$ is the image of the intersection of $F_i$ with $G_i$. Note that in fact the final Cremona transformation
contracts the two reducible (-1)--cycles on $V_2$ formed by $C_i+G_i$ to points on $L$ which are exactly the points $v_i$, $i=1,2$. 

The configuration formed by $\Gamma=\Gamma_1+\Gamma_2$ and $L$ is completely general, because making the inverse Cremona transformation we arrive at the system  $\L_{\tilde V_3}$, whose base points are general with the only constraint of the two infinitely near ones.

Finally, let us remark that there is a one parameter family $\mathcal F$ of projective transformations of the plane
mapping $u_1$ to $u_2$, $v_1$ to $v_2$ and $\Gamma_1$ to $\Gamma_2$. The induced transformations between $\Gamma_1$ and $\Gamma_2$ correspond to the projective 
transformations between $F_1$ and $F_2$ mapping $p_1$ to $p_2$ and $q_1$ to $q_2$.
Let $\omega\in \mathcal F$ be general. The points $s_i=\omega(a_i)$ lie on $\Gamma_2$ and the points $t_i=\omega^ {-1}(a_i)$ on $\Gamma_1$, $i=1,\dots,4$. On the whole we have eight points $\{s_i, t_i\}_{1\leq i\leq 4}$ forming a divisor  $D_\omega$ on the curve $\Gamma$.

\begin{lemma}\label{lem:ess} In the above setting, for general $\omega\in \mathcal F$, neither $D_\omega$ is cut out on $\Gamma$ by a conic nor $2D_\omega$ is cut out on $\Gamma$ by a quartic.
\end{lemma}

\begin{proof} As for the first assertion, suppose that for the general $\omega\in \mathcal F$ there is conic $\Gamma_\omega$ cutting out $D_\omega$ on $\Gamma$. For some special $\omega$, one of the $s_i$ points, say  $s_1$, coincides with one of the points $a_1,\dots,a_4$, whereas $s_2,s_3, s_4$ do not lie on $\Gamma_1$. By the generality assumption, for this $\omega$ the points $t_1,\dots,t_4$ stay distinct from $a_1,\dots,a_4$. Then the conic $\Gamma_\omega$ must coincide with $\Gamma_1$, since it has five points in common with it. But then it does not contain 
$s_2,s_3, s_4$ which is a contradiction.
The proof of the second assertion is similar and can be left to the reader. \end{proof}

\subsection{\bf Transversality.}\label{sec:transv}
Still referring to the second, or third, degeneration, consider the surface $V_2$, or $V_3$, which we will denote by $V$ here, and its normalization $\tilde V$. On $\tilde V$ we have the two curves $F_1, F_2$ 
which are glued via the correspondence $\omega$ to form $V$. Consider the linear system $\L_{\tilde V}$, and let $r$ be the dimension of the linear series $\mathcal R_F$ it cuts out on $F=F_1+F_2$, and $d=\L\cdot F_i$, $i=1,2$. We denote by $\mathcal R_{F_i}$ the linear series cut out by $\L_{\tilde V}$ on $F_i$, $i=1,2$, which both have degree $d$ and dimension $r$.
Then we can consider  following two subvareties of ${\rm Sym}^ d(F_1)\times {\rm Sym}^ d(F_2)$:
\begin{itemize}
\item $X$, of dimension $r$, consisting of all pairs $(D_1,D_2)$ such that $D_1+D_2\in \L_F$;
\item $Y$, of dimension  $d$, consisting of all pairs $(D_1,D_2)$ such that $D_2=\omega(D_1)$.
\end{itemize}

One has the following transversality statement:

\begin{lemma}\label {lem:transv} In the above setting, for general choices, $X$ and $Y$ intersect properly inside ${\rm Sym}^ d(F_1)\times {\rm Sym}^ d(F_2)$, i.e. $\dim(X\cup Y)=\max \{r-d, -1\}$,
unless either the point $p_i$, or the point $q_i$ is inflectional for $\mathcal R_{F_i}$, for both $i=1,2$.
\end{lemma}

\begin{proof} The proof is the same as the one of Proposition 3.1 of  \cite {CM98} and therefore we do not dwell on it here. \end{proof}

A similar lemma holds for $A_1$ and $A_2$ on $Z$.

\section{The fourth degeneration: throwing the quartics}
\label{sec:4}

Now we want to analyze the situation when the ratio $d/m$
is at most $19/6$.
We perform the same $2$-throws, up through the third
degeneration, 
and then make our fourth degeneration by
throwing four curves in $V_3$ corresponding to certain quartics.

In what follows we assume that $174/55 \leq d/m \leq 19/6$.
We will use the same notation as above. In particular $\alpha=d-3m$.
We will additionally set 

$$\ell = -\mu=19m-6d$$
which is non--negative and 
$$\ell = 2r-s, \quad {\rm with}\quad  s \in \{0,1\}.$$

Note that 

$$d=3\ell+19\alpha,\quad m=\ell+6\alpha$$
and  $d/m\geq 174/55$ is equivalent to

$$\alpha\geq 9\ell.$$

Recall that in the third degeneration we have five surfaces:
\begin{itemize}
\item $V_3$, with linear system of the form
$\L(9a-2\ell;(4a-\ell)^4,[2a,2a]^2)$;
\item $Z_3$, with linear system of the form
$\L(10\alpha-6a;6\alpha-3a,(3\alpha-2a)^6, [b-2a,b-2a-e]^2)$;
\item $T_3$, with linear system of the form
$\L(10m-3d-2a; [b-2a,b-2a-e]^2)$;
\item $U_{1,3}$ and $U_{2,3}$, planes, with linear systems of degree $e$.
\end{itemize}

We now note that each of the four disjoint curves on $V_3$ corresponding to the quartics
\[
\L(4; 2, 2, 2, 1, [1,1]^2),\quad 
\L(4; 2, 2, 1, 2, [1,1]^ 2),\quad
\L(4; 2, 1, 2, 2, [1,1]^ 2),\quad 
\L(4; 1, 2, 2, 2, [1,1]^ 2)
\]
we already met in \S \ref {subsec:two} has intersection number
$-\ell$ with the system on $V_3$.
Each meets the double curve in two points,
along $F_1$ and $F_2$,
which the reader will recall are identified
(giving the self--double curve of $V_3$)
in the second degeneration.

We perform a $2$-throw for each of them,
resulting in our fourth degeneration.
This will consist of nine surfaces:
\begin{itemize}
\item $V_4$, the transform of $V_3$.
The normalization $\tilde{V_4}$ of $V_4$
is obtained from $\tilde{V_3}$, with an additional eight double blow--ups,
four each on the curves $F_i$,
corresponding to eight $[r,r-s]$-points, and by blowing down the 
curves corresponding to the four quartics.The four 2--throws
results in removing each quartic $\ell$ times from the bundle on
$V_4$, which therefore corresponds to a linear system of
 the form
$\L(9a-18\ell, (4a-8\ell)^4, [2a-4\ell,2a-4\ell]^2, [r,r-s]^8)$;
\item $Z_4$, $T_4$, $U_{1,4}$, and $U_{2,4}$,
unchanged from the corresponding surfaces in the third degeneration,
with the same bundles;
\item four new planes $Y_{i}$, $i=1,\ldots,4$, with bundles
of degree $s$ on them; these planes arise from the four $2$-throws.
\end{itemize}

The picture of the central fibre of the fourth degeneration is shown in
Figure \ref {fig:7}.

\begin{figure}
\begin{center}
\includegraphics [scale=0.70, width=15cm] {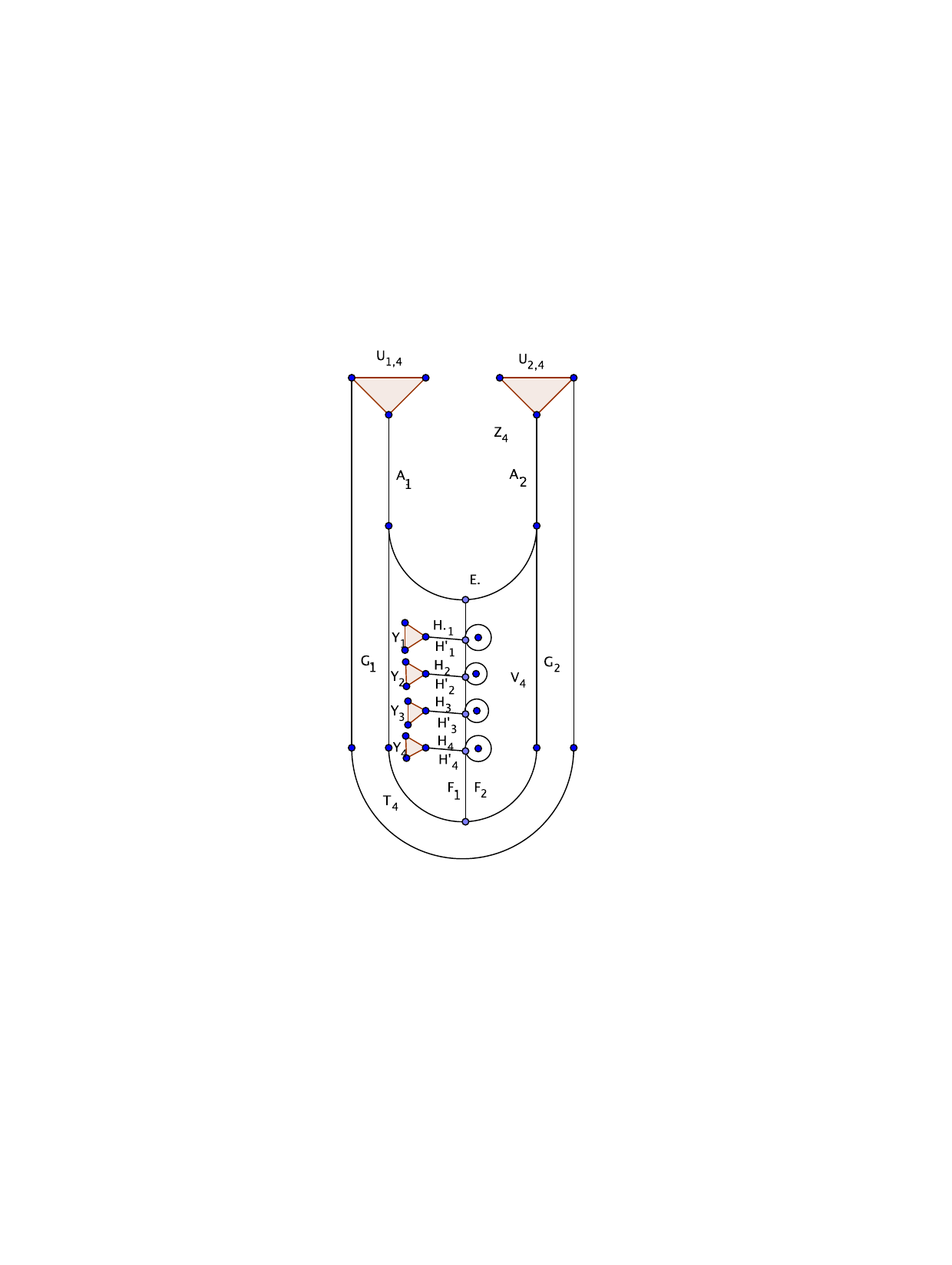}
\end{center}
\caption{the central fibre of the fourth degeneration}\label{fig:7}
\end{figure}

Note that, in order for the system on $V_4$ to be non--empty
it is necessary that $a \geq 2\ell$.

\begin{lemma} \label {lem:V4a}
If $a >  4\ell+1$ then the system on $V_4$ is non--empty and non--special.
If $\ell\leq 2$ the same holds if $a \geq 4\ell+1$.
\end{lemma}

\begin{proof}
For the system on $V_4$,
we have the following series of quadratic transformations which do not involve the
eight $[r,r-s]$--points:
{
\footnotesize

\[
\begin{matrix}
9(a-2\ell);& \underline{4(a-2\ell)},&\underline {4(a-2\ell)},
& \underline{4(a-2\ell)}, & 4(a-2\ell),
& [2(a-2\ell), 2(a-2\ell)],
&[ 2(a-2\ell), 2(a-2\ell)]\cr
6(a-2\ell);&a-2\ell,&a-2\ell,
& a-2\ell, & \underline {4(a-2\ell)},
& [\underline {2(a-2\ell)}, 2(a-2\ell)],
& [\underline {2(a-2\ell)}, 2(a-2\ell)] \cr
4(a-2\ell); & a-2\ell, & a-2\ell,
& a-2\ell, & \;\underline {2(a-2\ell)},
& \underline {2(a-2\ell)}\;, & \underline {2(a-2\ell)} \cr
2(a-2\ell); & \underline{a-2\ell}, & \underline{a-2\ell},
& \underline{a-2\ell} & & & \cr
a-2\ell & & & & & & \cr
\end{matrix}
\]
}
Hence the system on $V_4$ Cremona reduces to 
$\L(a-2\ell;[r,r-s]^8)$. We will therefore 
prove that this system is non--empty and non--special, 
by working on the plane blown up at the 
eight compound points rather than on $\tilde {V_4}$.

As we saw in \S \ref {subsec:two}, 
under the previous series of quadratic transformations,
the two curves $F_i$ map to two conics $\Gamma_i$, $i=1,2$.
Each of them contains four of the $[r,r-s]$-points, i.e. the points 
$t_i\in \Gamma_1$ and $s_i\in \Gamma_2$, $i=1,\dots,4$.
The divisor $D$ of degree 8 on  $\Gamma$ consisting of all
these points is not in the bicanonical series of $\Gamma$
by Lemma \ref {lem:ess}.

The curve $\Gamma = \Gamma_1 + \Gamma_2$
is $1$-connected, with arithmetic genus three.
Restricting to $\Gamma$,
we have for each $i=0,\ldots,r-1$ the exact sequences
\begin{align*}
0 \to \L(a-2\ell-4(i+1); & [r-i-1,r-s-i-1]^ 8)
 \to \L(a-2\ell-4i;[r-i,r-s-i]^ 8) \\
& \to  \L(a-2\ell-4i;[r-i,r-s-i]^ 8)_{\vert \Gamma} \to 0
\end{align*}
where, by abusing notation, we use the same symbol
to denote a linear system and the corresponding line bundle.

One has
$$\L(a-2\ell-4i;[r-i,r-s-i]^ 8) \cdot \Gamma =
4(a-4\ell) > 4$$
if $a>4\ell+1$, in which case
the restrictions of all the above line bundles
to $\Gamma$ have no $H^ 1$. 
Therefore,  the middle linear system
 is non--empty and non--special if the system on the left 
 is non--empty and non--special.
Since the system with $i=r-2$ is
$\L(a-4\ell-2s+4; [1,1-s]^8)$,
which is non--empty and non--special, we conclude by induction. 

Let now $a=4\ell+1$. If $\ell=0$ the system is clearly non--special.
If $1\leq \ell\leq 2$, then $r=1$ and $i=0$. Then
$\L(2\ell+1;[r,r-s]^ 8)_{\vert \Gamma}$ has degree 4.  However, by Lemma \ref {lem:ess}, 
$\L(2\ell+1;[r,r-s]^ 8)_{\vert \Gamma}$ is still not special and the same argument as
above can be applied to conclude.\end{proof}

Next we deal with the system on $Z_4$,
of the form 
$\L(10\alpha-6a;6\alpha-3a,(3\alpha-2a)^6, [b-2a,b-2a-e]^2)$,
which is the same as  the one on $Z_3$.
We saw in (\ref{Z2Cremona}) 
that this system is Cremona equivalent to
$\L(4\alpha-3a; (\alpha-a)^6, [b-2a,b-2a-e]^2)$ and
the ten multiple points
all lie on the irreducible nodal cubic curve $E$.
Hence we are in the strong anticanonical case.

We first note that $4\ell < \alpha - 2\ell\leq  (69d-218m)/2$.
The first inequality is equivalent to $d/m > 117/37$ and we have
$174/55 > 117/37$. The second inequality is equivalent to
$d/m\geq 136/43$ and we have also $174/55 > 136/43$. 

\begin{lemma}\label{lem:Z4}
If $174/55 \leq d/m \leq 19/6$, and $4\ell < a \leq \alpha-2\ell$,
then the linear system on $Z_4$ is non--empty and non--special.
\end{lemma}  

\begin{proof}
One has
\[
\L_{Z_4} \cdot K_{Z_4} = 
2\ell-a < 0.
\]
Next we note that $b-2a > \alpha - a$,
since this is equivalent to
$d/m < 16/5$.
Hence the three largest multiplicities of the  system on $Z_4$ are
$b-2a$, $b-2a$, and $b-2a-e$.
Comparing their sum $3b-6a-e$ with the degree $4\alpha-3a$, we see that
$3b-6a-e > 4\alpha-3a$, since this is equivalent to
$d/m < 54/17$, and $19/6 < 54/17$.
Hence the system is not standard, and we may apply a quadratic transformation
based at these three points.
This results in the system
  \[
\L(8\alpha-3b+e; (\alpha-a)^6, b-2a-e, (4\alpha+a-2b+e)^2, 4\alpha+a-2b)
\]
where we have suppressed the infinitely near nature of the points in the notation.
Now we have that
$\alpha-a \geq 4\alpha+a-2b+e$
since this is equivalent to
 $d/m \leq 19/6$.
So in this system, the three largest multiplicities are
$b-2a-e$, $\alpha-a$, and $\alpha -a$.
Their sum is at least the degree, since
$2\alpha + b - 4a-e \geq 8\alpha-3b+e$
is equivalent to
$d/m \leq 19/6$.
Perform another quadratic transformation centered at these three points;
we obtain the system
\[
\L(14 \alpha -7b + 4a+3e;
(\alpha-a)^4, 6\alpha +2a-3b+e, (7\alpha + 3a -4b + 2e)^2,
(4\alpha+a-2b+e)^2, 4\alpha+a-2b).
\]
  
  The three multiplicities
$\alpha-a$, $\alpha-a$, and $6\alpha +2a-3b+e$
have sum at least the degree
since $d/m \leq 19/6$;
another quadratic transformation
results in the system
 \[
\L(20\alpha -11b + 8a + 5e;
(\alpha-a)^2, 
12 \alpha -7b + 6a+3e,
(7\alpha + 3a -4b + 2e)^4,
(4\alpha+a-2b+e)^2, 4\alpha+a-2b
).
\]
Again the three multiplicities
$\alpha-a$, $\alpha-a$, and $12 \alpha -7b + 6a+3e$
have sum at least the degree
 since $d/m \leq 19/6$;
one further quadratic transformation
gives us the system
\[
\L(26\alpha -15b +12a+7e;
(7\alpha + 3a -4b + 2e)^6,
(4\alpha+a-2b+e)^2, 4\alpha+a-2b,
18\alpha - 11b + 10a+5e
).
\]
The multiplicity $18\alpha - 11b + 10a+5e$
is non--negative, since this is equivalent to
$2a \leq 69d-218m$.

Note that
\begin{align*}
26\alpha -15b +12a+7e =
48d-153m-3a+c &= c-m-3a-8\ell \\
18\alpha - 11b + 10a+5e =
34d - 109m - a + c &= c -a- m -2\alpha-6\ell \\
7\alpha + 3a -4b + 2e =
13d-41m -a &= \alpha-2\ell-a \\
4\alpha + a - 2b + e = 
7d-22m+a &= \alpha-\ell-a,
\end{align*}
We now observe that if $d/m \geq 174/55$, then
$c -a- m -2\alpha-6\ell \geq \alpha-\ell-a$
and the three largest multiplicities here are
\[
18\alpha - 11b + 10a+5e, 4\alpha + a - 2b + e, 4\alpha + a - 2b + e;
\]
since their sum is
$c -a- m -2\alpha-6\ell + 2(\alpha-\ell-a)
= c-m-3a-8\ell$
which is the degree of the system, the system is standard,
and excellent,
and we are done.
\end{proof}

Note that, in the above proof, if $d/m < 174/55$, then the system is not standard.

The system on $T_4$ is the same as that on $T_3$,
so we have the same criterion as in Lemma, i.e. we have
the system non--empty and non--special if $10m-3d \geq 2a$. 

\begin{lemma}\label{lem:T4}
If $d/m \leq 19/6$ and $a \leq \alpha - 2\ell$
then the system on $T_4$ is non--empty and non--special.
\end{lemma}

\begin{proof}
We must show that the hypothesis implies $2a \leq 10m-3d$,
which will be the case if  $2\alpha -4\ell \leq 10m-3d$.
This is equivalent to $d/m \leq 92/29$,
which is true, since $19/6 < 92/29$.
\end{proof}

We deal now with checking the surjectivity criteria 
from Proposition \ref {prop:fine}. In this case we set

$$W_1=V_4, W_2=W_1+Z_4, W_3=W_2+T_4, W_4=W_3+U_{1,4}+U_{2,4}, W_5=X_0.$$

Since $W_1$ is again non--normal,
we have to deal with the self--double curves.
We abuse notation,
and still denote by $F_i$, $i=1,2$, and $E$
the proper transform on $\tilde V_4$ of these curves on $\tilde V_3$.
Let $H_1,\ldots, H_4$ be the second exceptional divisors
of four of the second sets of eight compound singularities,
the ones that meet $F_1$.
Let $D=F_1+H_1+\ldots+H_4$.
With this notation,
$V_4$ is obtained by suitably gluing $D$ with its counterpart $D'$
formed by the proper transform of $F_2$
and the four curves $H'_1,\ldots,H'_4$
analogous to $H_1,\ldots, H_4$.

\begin{lemma}\label {lem:V4F}  Let $174/55 \leq  d/m \leq 19/6$. Then:
\begin{itemize}
\item [(i)]  if $a>4\ell+3$,
then  $H^1(\tilde V_4,\L_{\tilde V_4}(-D-D'))=0$. If $\ell=2$
the same conclusion holds if $a\geq 11$.
If $0\leq \ell\leq 1$  the same conclusion holds for all $a$;
\item [(ii)] if  $a>4\ell +4$,
then $H^1(\tilde V_4,\L_{\tilde V_4}(-D-D'-E))=0$.
 If $\ell=2$
the same conclusion holds if $a\geq 12$.
If  $0\leq \ell\leq 1$  the same conclusion holds for all $a$.
\end{itemize}
\end{lemma}

\begin{proof} 
To prove (i),
one checks that the system $\L_{\tilde V_4}(-D-D')$
corresponds to
$\L(9(a-2\ell);
(4(a-2\ell))^4,[2(a-2\ell)+1, 2(a-2\ell)]^ 2, [r-s,r-1]^ 8)$. 
There are six $(-1)$--curves which intersect this system negatively,
namely  the four quartics we threw in this section and the two conics
we threw in the third degeneration. 
The intersection number of the system with the conics is $-1$, whereas the intersection with
the quartics is $-2$. The conics therefore split once and 
the quartics twice.
However all these $(-1)$--curves have been contracted on $\tilde V_4$, so 
$\L_{\tilde V_4}(-D-D')$ actually corresponds to the system obtained from
$\L(9(a-2\ell);
(4(a-2\ell))^4,[2(a-2\ell)+1, 2(a-2\ell)]^ 2, [r-s,r-1]^ 8)$ by removing these $(-1)$--curves
which appear in its fixed part.
This residual system is Cremona equivalent to
$\L(a-2\ell-4; [r-s,r-1]^ 8)$. The case $0\leq \ell\leq 1$  is obvious. 
Otherwise, if $a>4\ell+3$ 
we note that
the intersection with the curve $\Gamma$
which we considered in the proof of Lemma \ref{lem:V4a} is $4(a-4\ell-2)>4$
and the argument of that Lemma applies directly to
conclude the proof
 of the first assertion. If $\ell=2$ the intersection with the curve $\Gamma$
 has degree 4, but it is not special by Lemma \ref {lem:ess} and we finish as before.

As for (ii), note that, as above, $\L_{\tilde V_4}(-D-D'-E)$
corresponds to a system which is Cremona equivalent to 
$\L(a-2\ell-5; [r-s,r-1]^ 8)$. Again, the case $0\leq \ell\leq 1$  is clear.
Otherwise note that the intersection with the curve $\Gamma$ is $4(a-4\ell-3)>4$
and  the argument of Lemma \ref{lem:V4a} applies again to
conclude the proof in this case too. For $\ell=2$ the same argument as above applies. \end{proof}

Going back now to the surface $W_1=V_4$,
we want to apply the criterion for non--speciality as in Remark \ref {rmk:fu}.
If $C$ is the double curve of $V_4$, then its pull-back on $\tilde V_4$ is 
the reducible curve $D+D'$. The surjectivity
$H^ 0(\tilde V_4,\L_{\tilde V_4})\to H^ 0(D+D',\L_{D+D'})$ 
is a consequence of Lemma \ref  {lem:V4F}, if either
$a>4\ell+3$ or $a\geq 11$ and $\ell=2$ or $\ell\leq 1$.
The double curve $C$ 
is gotten from $D$, or $D'$, by gluing four
pairs of points, hence $C$ has arithmetic genus 4. Since $\L_{D}$ has degree
at least 8, then $\L_C$ is non special. 
Thus the criterion in Remark \ref {rmk:fu} can be applied and 
the system is non--special if either $a>4\ell+3$ or $a\geq 11$ and $\ell=2$ or $\ell\leq 1$.

Now, in passing from $W_1$ to $W_2$ we glue the surface $Z_4$ to $W_1$,
along the proper transforms of the curve $E$. 
We want to apply the criterion for non--speciality given by 
Proposition \ref{prop:fine}. Condition (i) of that criterion is verified
by Lemmas \ref {lem:V4a} and \ref {lem:Z4} as soon as $a\geq 4\ell+2$. 
Condition (ii) is a consequence of $H^ 1(V_4,\L_{V_4}(-E))=0$. 
An obvious variation of Remark \ref {rmk:fu} shows that this is in turn a consequence of
$H^1(\tilde V_4,\L_{\tilde V_4}(-D-D'-E))=0$, which we have as soon as
either $a>4\ell+4$, or $a\geq 12$ and $\ell=2$, or $\ell \leq 1$. 

Therefore, if either  $a>4\ell+4$  or $a\geq 12$ and $\ell=2$, or $\ell \leq 1$ and $a\geq 4\ell+1$
also the system on $W_2$ is 
non--special. 
Furthermore the restriction map
$H^ 0(W_2,\L)\to H^ 0(Z_4,\L_{Z_4})$ is surjective.

Next we glue $T_4$, creating $W_3$, and to do this we will choose
a specific value of $a$.
The double curve is $\Delta=A_1+A_2+G_1+G_2$.
Again we apply Remark \ref {rmk:fu}.

The surjectivity there is implied by the following:

\begin{lemma}\label {lem:Z4F}
Assume $a>2\ell-2$, $174/55 \leq  d/m \leq 19/6$ and take 
$a=\alpha-2\ell-h$, with:

\begin{itemize}
\item [(i)]  $h=1$ if either:

\begin{itemize}
\item [(a)] $\ell\geq 4$ or
\item [(b)]  $\ell\leq 3$ and $\alpha \geq 7\ell+7$;
\end{itemize}

\item  [(ii)]   $h=2$ if either:
\begin{itemize}
\item [(a')] $\ell=3$ and $\alpha=27$ or
\item [(b')] $\ell\leq 2$ and $7\ell+5\leq \alpha\leq 7\ell+6$;
\end{itemize}

\item  [(iii)]  $h=3$ if $1\leq \ell\leq 2$ and $\alpha = 7\ell+4$.

\end{itemize}
Then  $H^1(Z_4,\L_{ Z_4}(-A_1-A_2))=0$ and, in the same hypotheses as in part (ii) of Lemma
\ref   {lem:V4F},  
the restriction map $H^ 0(W_2,\L_{W_2})\to H^ 0(\Delta,\L_\Delta)$ is surjective.
\end{lemma} 

\begin {proof} Suppose for a moment we proved that  $H^1(Z_4,\L_{ Z_4}(-A_1-A_2))=0$. 
Look at the commutative diagram

$$
\begin{matrix}
&&&\L_{G_1+G_2}(-A_1-A_2)\\
&&&\downarrow\\
&\L_{W2}&\longrightarrow&\L_\Delta\\
&\downarrow&&\downarrow\\
&\L_{Z_4}& \longrightarrow&\L_{A_1+A_2}
\end{matrix}
$$
Notice that $H^ 0(G_1+G_2,\L_{G_1+G_2}(-A_1-A_2))\simeq H^ 1(G_1+G_2,\L_{G_1+G_2}(-A_1-A_2))= 0$
by degree reasons. Thus $H^ 0(\Delta, \L_\Delta)\simeq H^ 0(A_1+A_2,\L_{A_1+A_2})$. Since
the map $H^ 0(W_2,L_{W_2})\to H^ 0(Z_4, \L_{Z_4})$ is surjective by part (ii) of Lemma
\ref   {lem:V4F},  we have that also $H^ 0(W_2,\L_{W_2})\to H^ 0(\Delta,\L_\Delta)$ is surjective.

The proof that  $H^1(Z_4,\L_{ Z_4}(-A_1-A_2))=0$
parallels the one of Lemma \ref {lem:Z4}. The system $\L_{ Z_4}(-A_1-A_2)$ 
is Cremona equivalent to $\L(4\alpha-3a; (\alpha-a)^6, [b-2a+1-e,b-2a]^2)$ and we are in the
strong anticanonical case.
Moreover  the intersection of $\L_{ Z_4}(-A_1-A_2)$ with the canonical bundle is
$2\ell-a+2<0$. 

First make
a quadratic transformation based at the points of multiplicities
$b-2a+1-e, b-2a+1-e,b-2a$, getting the linear system
\[
\L(8\alpha-3b+2e-2;(\alpha-a)^ 6,b-2a, (4\alpha+a-2b-1+e)^ 2, 4\alpha+a-2b-2+2e).
\]
Next make a quadratic transformation based at points of multiplicities
$\alpha-a,\alpha-a, b-2a$, getting the linear system
\[
\L(14\alpha-7b+4a+4e-4;(\alpha-a)^ 4,(7\alpha+3a-4b+2e-2)^ 2,
\]
\[
6\alpha+2a-3b+2e-2, (4\alpha+a-2b-1+e)^ 2, 4\alpha+a-2b-2+2e).
\]
Again, make a quadratic transformation based at points of multiplicities
$\alpha-a,\alpha-a, 6\alpha+2a-3b+2e-2$, getting the linear system
\[
\L(20\alpha-11b+8a+6e-6; (\alpha-a)^ 2, 12\alpha+6a-7b+4e-4,
\]
\[
(7\alpha+3a-4b+2e-2)^ 4,  (4\alpha+a-2b-1+e)^ 2, 4\alpha+a-2b-2+2e).
\]
Then, make a quadratic transformation based at the points of multiplicities
$\alpha-a,\alpha-a, 12\alpha+6a-7b+4e-4$, getting the linear system
\[
\L(26\alpha-15b+12a+8e-8;(7\alpha+3a-4b+2e-2)^ 6,
\]
\[
18\alpha-11b+10a+6e-6,  (4\alpha+a-2b-1+e)^ 2, 4\alpha+a-2b-2+2e).
\]

Set 
\[
x= 18\alpha-11b+10a+6e-6, y=4\alpha+a-2b-1+e, z=4\alpha+a-2b-2+2e.
\]
Note that $y=z-e+1$, that
\[
26\alpha-15b+12a+8e-8=x+y+z-e+1=x+2y
\]
and
\[
7\alpha+3a-4b+2e-2=z-\ell-e
\]
so that the system may be written as
\[
\L(x+2y;x, y^ 2,z,(z-\ell-e)^ 6).
\]

Moreover one has
$z-\ell-e=h-2\geq -1$, $z=\ell+h+e-2\geq \ell+e-1\geq -1$, and
$y=z-e+1=\ell+h-1\geq \ell\geq 0$.
In addition one computes
$$x=\frac {\alpha-7\ell+e}2-6+h;$$
therefore  $x\geq -1$ if and only if $\alpha+e\geq 7\ell+10-2h$,
which is equivalent to $\alpha\geq 7\ell+9-2h$ (if equality holds
one has $e=1$). One checks that the hypotheses relating $\ell, \alpha$ and 
$h$ ensure that $x\geq -1$.

Suppose we are in case (i) or (ii). 
Then after splitting  single $(-1)$--curves with no contribution to $H^ 1$,
the residual system is
$\L(x+2y;x,y^ 2,z)$, which is non--empty and non--special.

In case (iii) the relevant system is $\L(2\ell+3; (\ell+2)^ 2, \ell+1,1^ 6)$, 
which is  non--empty and non--special in our cases.
\end{proof}

Note that the previous lemma applies in all cases but six, which are 

$$\ell=0, 1\leq \alpha\leq 4, \quad {\rm corresponding}\quad {\rm to}\quad (d,m)=(19\alpha,6\alpha)$$
$$\ell=1, 9\leq \alpha\leq 10, \quad {\rm corresponding}\quad {\rm to}\quad (d,m)=(3+19\alpha,1+6\alpha).$$

Finally we attach the planes $U_{1,4}$, and $U_{2,4}$
and  $Y_{i}$, $i=1,\ldots,4$,
in order to obtain first $W_4$ and then $W_5$.
In this case  the non--speciality is clear 
since the bundles there have degree $e$. 

\begin{proposition}
\label{174/55<=ratio}
If $174/55 \leq d/m \leq 19/6$ and $(d,m)\neq (174,55), (193,61), (348, 110)$,
then the system $\L(d;m^{10})$
has the expected dimension.
\end{proposition}

\begin{proof} Suppose $(d,m)$ correspond to a pair $(\alpha,\ell)$ satisfying any of the
hypotheses of Lemma  \ref {lem:Z4F}. Then we can take $a=\alpha-2\ell-h$ as indicated 
by the lemma.  If $\ell=0,1$
we need $a\geq 4\ell+1$ in order to meet all the 
hypotheses of the previous lemmas.  
This requires $\alpha\geq 6\ell+1+h$
which is true in all cases, and one concludes the proof using the results of \S \ref {sec:computation}.
If $\ell=2$ we need $a\geq 12$, and $\alpha-4-h\geq 12$ in all cases, except for $(d,m)=(348, 110)$,
i.e. $\alpha=18$. 
If $\ell>2$ we need $a> 4\ell+5$, and $\alpha-2\ell-h>4\ell+5$ in all cases. 

Finally the cases $(d,m)=(19\alpha,6\alpha)$ with $\alpha\leq 4$ are
covered by the results in \cite {CM00}, \cite  {CCMO}, \cite {dumnicki2} and cite {dumnicki1}. \end{proof}

\section{The remaining cases}\label{sec:174}

In this section we prove non--speciality of the three bundles 
left out in Proposition \ref  {174/55<=ratio} and this will finish the proof of 
Theorem \ref  {thm:main}.

\subsection{\bf The case d=174, m=55.}
\label{subsec:174}

The non--speciality of the  linear system $\L(174;55^ {10})$
does not follow directly from the above arguments,
and must be proved with an ad hoc argument,
which however uses the above setting. In particular we  make  
an analysis with the third degeration (see \S \ref {sec:3}).

Note that the virtual dimension is $-1$,
and therefore we have to prove that the system is empty.
We argue by contradiction and suppose this is not the case.
Then, by fixing extra base points,
we may assume that the dimension of the linear system
is zero on the general fibre of our degeneration.
Thus the curve on the general fibre describes a surface
in the total space of the degeneration, excluding the central fiber.
By taking the closure of this surface,
one finds a surface $S$ which intersects the central fibre along a curve.
We then conclude that there is some line bundle $\L$
on the total space of the degeneration,
i.e. the line bundle determined by $S$,
which is a limit line bundle of $\L(174;55^ {10})$ on the central fibre,
and such that the general section of the restriction of $\L$ to the central fibre
does not vanish identically on any irreducible component of the central fibre.
We will then say that the limit line bundle $\L_{\vert X_0}$ is \emph{centrally effective}. 

In particular, in the setting of \S \ref {sec:3},
there must be an integer $a$
for which the corresponding limit line bundle is centrally effective.
We will prove emptiness of $\L(174;55^ {10})$
by showing that there is no such an $a$. 

For $d=174$ and $m=55$, we have
$c=87$, $e=0$, $\alpha = 9$, $\ell =r=s= 1$, and $b=a+14$.

In the third degeneration, this gives the following bundles:
\begin{itemize}
\item $\L_{\tilde{V_3}} = \L(9a-2;(4a-1)^4,[2a,2a]^2)$;

\item $\L_{Z_3} = \L(36-3a;(9-a)^6, [14-a,14-a]^2)$;

\item $\L_{T_3} = \L(28-2a; [14-a,14-a]^2)$;

\item $\L_{U_{i,3}} = \O$ for $i=1,2$.
\end{itemize}

On $V_3$, the four quartics
\[
\L(4; 2, 2, 2, 1, [1,1]^2),\quad
\L(4; 2, 2, 1, 2, [1,1]^2),\quad
\L(4; 2, 1, 2, 2, [1,1]^2),\quad
\L(4; 1, 2, 2, 2, [1,1]^2)
\]
each have intersection number $-1$ with the system $\L_{\tilde{V_3}}$,
and so split off once each;
the residual system is
$\L'_{\tilde{V_3}} = \L(9a-18;(4a-8)^4,[2a-4,2a-4]^2)$.

The four quartics each meet the self-double curve twice
(once each on each side) and no other double curves on $V_3$.

Recalling the proof of Lemma \ref  {lem:V4a}, we see
that the system  $\L'_{\tilde{V_3}} $
is Cremona equivalent to the system of curves of degree $a-2$.
In order to give a divisor on $V_3$,  the curves
in the system
must match with the four quartics;
this requires that  they meet the self--double curve
in the eight fixed points that the four quartics do,
four on each side. 

Note that the self--double curve, which on $\tilde V_3$ pulls back to $F_1+F_2$, is Cremona equivalent to the curve $\Gamma=\Gamma_1+\Gamma_2$
we met in \S \ref {subsec:two}. The
eight points correspond to the eight points $t_i, s_i$, $i=1,\dots,4$, on $\Gamma$, four on each 
of the conics $\Gamma_1$ and $\Gamma_2$. We denote by $D$
the divisor formed by these eight points on $\Gamma$. Lemma \ref
{lem:ess} says that $D$ is
not a bicanonical divisor on $\Gamma$. There is no curve of degree
$a-2$ containing $D$ unless $a\geq 5$ and, for $a\geq 5$, $D$
gives  eight independent conditions to $\L(a-2)$. 
Hence we must have $a \geq 5$ for the limit bundle to be centrally effective,
in particular effective on $V_3$.

For $a=5$ the hypothesis in the transversality  Lemma \ref {lem:transv} is still
met. Indeed  $\L(3)(-D)$ is a pencil of cubics. What one has to show is that
it does not cut out on $\Gamma_1$ or $\Gamma_2$, off the base points in $D$,
a pencil of degree 2  with ramification points at $p_1, p_2$ or $q_1,q_2$. 
Suppose this is the case for all choices of the projective transformation
$\omega$ as in \S \ref {subsec:two}. 
Then, as in the proof of Lemma \ref {lem:ess},  for some special $\omega$, one of the points $s_1,\dots, s_4$ coincides with one of the points $a_1,\dots,a_4$. Then $\Lambda$ has a base point there, which is different from $p_1$ and $q_1$, a contradiction. 
Therefore, if $a=5$, this transversality implies that there are no divisors in $\L(3)(-D)$
which may correspond to Cartier divisors on $V_3$. 
In conclusion we must have $a \geq 6$ for central effectivity.

The system on $Z_3$ is Cremona equivalent to
$\L(36-3a; (9-a)^6, [14-a,14-a]^2)$
and, after the Cremona transformations,
the two compound points are not collinear.
Thus if $28-2a > 36-3a$ (i.e., if $a > 8$)
the two tangent lines to $E$ at these two points split.
If so, the system is empty,
since the two lines become $(-1)$--curves which meet.
Thus we must have $a\leq 8$ for the limit bundle to be centrally effective,
in particular effective on $Z_3$.

The system on $T_3$ corresponds to a pencil of conics, which,
as we saw in \S \ref {subsec:one},  sets up a correspondence between the
two curves $A_1$ and $A_2$, the intersection of $T_3$ with $Z_3$. 
The divisors in $\L_{Z_3}$ which correspond to divisors on the central
fibre must cut $A_1$ and $A_2$ in corresponding points. 

Return now to the linear system on $\tilde V_3$ equivalent to $\L(a-2)(-D)$, 
formed by all curves of degree
$a-2$ containing $D$. 
For $a\geq 6$ this system cuts out on either one of the conics $\Gamma_1$ and $\Gamma_2$
a complete linear series. This means that on $\tilde V_3$ the linear system 
$\L_{\tilde V_3}$ minus the 
four quartics and passing through the points on $F_1$ and $F_2$
corresponding to $t_i,s_i$, $i=1,\dots,4$, cuts out on either one of the curves $F_1$ and $F_2$,
off these base points, a complete linear series. 
Then we may apply Lemma  \ref {lem:transv}, and conclude that curves in 
$\L(a-2)(-D)$ which may correspond to Cartier divisors on $V_3$ 
form a system of dimension 
$(a-2)(a+1)/2 - 8 - (2a-8) = (a^2-5a-2)/2$.

Suppose we can apply transversality as indicated in \S \ref {sec:transv}.
Then  there are $14-a$ additional conditions on the system
on $Z_3$ in order to obtain divisors which  match with curves on $T_3$
and hence have a chance to be Cartier on the central fibre.
Thus the dimension $\delta$ of the family of these divisors is at most

\begin{equation}\label {eq:dimen}
\delta = \max \{-1, \dim( \L(36-3a; (9-a)^6, [14-a,14-a]^2))-14+a\}
\end{equation}
and central effectivity requires $\delta\geq 0$. 

We note that the curves of this system, as well as the ones on $V_3$,
restrict to the double curve $E=V_3\cap Z_3$
(which on $Z_3$ in this form is a nodal cubic through the six $(9-a)$-points)
in degree $a-2$.

We take up the three cases for $a$ in turn and verify the needed hypothesis for
transervality.

If $a=8$ the system on $Z_3$ is Cremona equivalent to
$\L(12; 1^6, [6,6]^2)$. By making the obvious Cremona transformations,
it can be further reduced to $\L(0;(-1)^ 6,0^ 4)$, which means that
it consists of six disjoint (-1)--curves. In this process the curves  $B_1$  and $A_1$ are
respectively mapped to a line $M$  through three of the last four points and to a cubic $N$
double at one of these points and passing through the first six base points. 
Since these six points are general, none of them coincides with the intersection
of $N$ and $M$ off the common base point. This suffices to apply Lemma \ref  {lem:transv}
and we conclude that the system is empty because then $\delta=-1$.

If $a=7$, the system on $Z_3$ is Cremona equivalent to $\L(3; 1^4)$, which 
has dimension 5. Again the generality assumptions provide the 
required hypothesis for the transversality and we conclude as above
since $\delta=-1$.

Finally let $a=6$.
By Lemma \ref{lem:V3F},
the linear system $\L_{V_3}$ on $V_3$ is
non--special of dimension $2$.
Since for $a=5$ the linear system on $V_3$ is empty,
then the kernel linear system $\L_{V_3}(-E)$ is empty,
and therefore the linear system cut out by $\L_{V_3}$ on $E$ has dimension $2$.

The line bundle on $Z_3$ is Cremona equivalent to
$\L(6;1^ 6,2^ 4)$, which is non special, of dimension $9$,
(see also by Lemma \ref {lem:Z4}). This cuts out a complete linear series 
on the cubic $N$ since the kernel linear system has dimension zero.
Hence again the transversality holds and we conclude that $\delta=1$ by
 \eqref  {eq:dimen}. 

The kernel linear system to $E$ is empty
(see the analysis of the $a=7$ case),
hence the linear system cut out on $E$ has dimension
$r\leq 1$.

There is no matching possible
between this $r$--dimensional linear system of degree $4$ on $E$,
and the dimension $2$ linear system restricted from the $V_3$ side,
by generality of the construction.
The curve $E$ is a nodal rational curve,
and we may embed its normalization into $\P^4$
as a rational normal curve of degree $4$;
we have the complete $g^ 4_4$ on it
which is cut out by the hyperplanes in $\P^ 4$.
Let $x_1$ and $x_2$ be the two points on the curve
that are identified to form the nodal curve $E$.
Consider the chordal line passing through the two points.
Each point on this line gives rise to a complete $g_4^ 3$
which identifies the two points.
The line minus the two points may be identified with the
natural $\C^*$ component of the Picard group of $E$.
The two sublinear systems that we have correspond to a line and a plane
meeting the chord,
giving rise to a $g_4^ 1$ and $g_4^ 2$ identifying the points.
The usual generality and transversality arguments as in \S \ref {sect:gentrasv},
imply that the above line and plane are general under the condition
of meeting the chord.
Hence there is no hyperplane containing these both,
and so there is no common restriction.

This completes the analysis of the $\L(174;{55}^{10})$ system,
and we conclude that it is empty as expected.

\subsection{\bf The case d=193, m=61.}
\label{subsec:193}

Here $\ell=1$ again and the virtual dimension is $4$. The analysis,
as in the previous case, uses the third degeneration and the four quartics 
split off once each. 

Using $a=7$  we see that:
\begin{itemize} 
\item the linear system on $\tilde V_3$ Cremona reduces to $\L(5;1^ 8)$, where the base points lie, as above, on $\Gamma$, which has dimension 12;
\item the linear system on $Z_3$ Cremona reduces to $\L(6;2^ 3, 1^ 7)$, which has dimension 11;
\item the linear system on $T_3$ is composed with 8 conics plus a fixed line.
\end{itemize}
All systems are non--empty and non--special, the transversality criteria all hold. Indeed,
from part (ii) of Lemma \ref {lem:V4F}, we see
that the system of curves 
on $V_4$ cuts out a complete linear series
on the double curve $E$. Moreover one directly verifies that  $\L_{Z_3}$  
cuts out complete linear series on the curves $A_i+B_i$, for $i=1,2$.
Therefore the dimension of the system on $V_3$ is 6. Similarly the dimension of the system on $Z_3+T_3$ is 2 and moreover its restriction of the latter system to $E$ is injective. The restriction of the $V_3$  system to $E$ has dimension 3. Using the same analysis as in the previous case, we see that  these two series on $E$ intersect in a linear series of dimension 1, and this implies that the family of matching curves on the central fibre has dimension 4.

\subsection{\bf The case d=348, m=110.}
\label{subsec:348}

This is the double of $\L(174;55^ {10})$, its virtual dimension is 
24, and we can prove it is non--special using the same ideas as
above. Actually the proof is easier and we will be brief.

We have
$c=174$, $e=0$, $\alpha = 18$, $\ell =2, r= 1, s=0$, and $b=a+28$.

Consider the fourth degeneration: this is necessary since the
quartics split with multiplicity 2 from the bundle on $V_3$ in the
third degeneration. Then
we have the following bundles (see \S \ref {sec:4}):
\begin{itemize}
\item $\L_{\tilde{V_4}}$, which is Cremona equivalent to
$\L(a-4;[1,1]^ 8)$ and the eight compound simple points
lie four each on the two conics $\Gamma_1,\Gamma_2$;

\item $\L_{Z_4}$ which is Cremona equivalent to $\L(72-3a;(18-a)^ 6, [28-a,28-a]^ 2$.
This is in turn Cremona equivalent to a system of the form 
$ \L(48-3a;(16-a)^4, (14-a)^6)$, where the six points of multiplicity
$14-a$ are general,
whereas the four points of multiplicity $16-a$ are not; three of them
are on the same line $L$ and two of these are infinitely near
giving a compound $[16-a,16-a]$ point; the line $L$ is the image
of one of the curves $G_i$, the other curve $G_{3-i}$ is contracted to the compound 
$[16-a,16-a]$ point;

\item $\L_{T_4}$ is Cremona equivalent to $\L(56-2a; [28-a,28-a]^2)$;

\item $\L_{U_{i,4}}$ and  $\L_{Y_{j}}$ fare trivial, for  $i=1,2$ and  $j=1,\dots,4$.
\end{itemize}

We fix $a=14$. Then $\dim(\L(10;[1,1]^ 8))=49$. In order to compute the dimension
of $\L_{V_4}$, we have to take into account the matching conditions with the self double
curves. The usual transversality arguments tell us that this imposes  at least 16 
conditions, i.e., 12 for the matching on $F_1$ and $F_2$ and 4 for the matching
on the curves $H_i, H'_i$, $i=1,\dots,4$  (see \S  \ref {sec:4}). Hence we find
$\dim (\L_{V_4})\leq 33$. In addition, from part (ii) of Lemma \ref {lem:V4F}, we see
that the system of matching curves 
on $V_4$ cuts out a complete linear series
on the double curve $E$.

The system  $\L_{Z_4}$ is Cremona equivalent to $\L(6;2^ 4)$. The base points
are not general, but the above description  implies that the system
is non--special, of dimension 15. The usual transversality gives us 14 more
matching conditions along the curves $A_1,A_2$, on each of which $\L_{Z_4}$ cuts out a complete
linear series. Hence the system on $Z_4+T_4$  has dimension 1.

We claim that the dimension of the system of matching curves on $V_4\cup Z_4\cup T_4$ is 24. Indeed such a curve 
can be constructed as follows: take any curve $D_1$ on $\L_{Z_4}$, which depends on one parameter,
and add any curve $D_2$ on $Z_4$ such that $D_1$ and $D_2$ cut out the same divisor
on $E$. Since $E$ is a rational curve of arithmetic genus 1,  $\L_{Z_4}\cdot E=\L_{V_4}\cdot E=10$, 
and the matching curves on $V_4$ cut out a complete linear series on $E$,
we see that $D_2$ varies with at most $23$ parameters. This proves our claim.

Finally, adding the remaining surfaces $U_{i,4}, Y_j$ for $i=1,2$ and $j=1,\dots,4$ 
does not increase the number of parameters of matching curves. This gives an upper bound 
of $24$ for the dimension of the system on the central fibre, thus proving non--speciality in this case.

\end{document}